\documentclass[a4paper,11pt]{amsart}
\usepackage[utf8]{inputenc}
\usepackage{amsthm,amsmath,amssymb,anysize,enumerate}
\usepackage[cmtip,all]{xy}
\usepackage{hyperref}
\usepackage{cite}
\marginsize{1.9cm}{1.9cm}{1.9cm}{1.9cm}

\newcommand{\ZZ}{\mathbb{Z}}
\newcommand{\PP}{\mathbb{P}}
\newcommand{\gm}{{\mathbb{G}}_m}
\renewcommand{\AA}{\mathbb{A}}
\newcommand{\kk}{\Bbbk}
\newcommand{\CC}{\mathbb{C}}

\newcommand{\bR}{\mathbf{R}}
\newcommand{\bL}{\mathbf{L}}

\newcommand{\cD}{\mathcal{D}}

\newcommand{\cG}{\mathcal{G}}
\newcommand{\cH}{\mathcal{H}}
\newcommand{\cK}{\mathcal{K}}
\newcommand{\cL}{\mathcal{L}}
\newcommand{\cM}{\mathcal{M}}
\newcommand{\cN}{\mathcal{N}}
\newcommand{\cO}{\mathcal{O}}
\newcommand{\cT}{\mathcal{T}}
\newcommand{\cX}{\mathcal{X}}
\newcommand{\cXn}{\cX_{n,w}}
\newcommand{\cXaf}{\mathcal{X}^{\text{af}}}
\newcommand{\cXinf}{\mathcal{X}^\infty}
\newcommand{\cY}{\mathcal{Y}}
\newcommand{\cYn}{\cY_{n,w}}
\newcommand{\cZ}{\mathcal{Z}}
\newcommand{\cZbar}{\bar{\mathcal{Z}}}
\newcommand{\cZinf}{\mathcal{Z}^\infty}

\newcommand{\xx}{\underline{x}}
\newcommand{\fia}{\varphi_{\alpha}}
\newcommand{\fiai}{\varphi_{\alpha}^{-1}}
\newcommand{\D}{\operatorname{D}^{\operatorname{b}}}
\newcommand{\Dc}{\operatorname{D}_{\operatorname{c}}^{\operatorname{b}}}

\newcommand{\Drh}{\operatorname{D}_{\operatorname{rh}}^{\operatorname{b}}}

\newcommand{\dd}{\partial}
\newcommand{\ra}{\rightarrow}
\newcommand{\lra}{\longrightarrow}
\newcommand{\hra}{\hookrightarrow}

\newcommand{\Hom}{\operatorname{Hom}}
\newcommand{\Ext}{\operatorname{Ext}}
\newcommand{\DR}{\operatorname{DR}}
\newcommand{\rk}{\operatorname{rk}}
\newcommand{\id}{\operatorname{id}}

\newcommand{\cancel}{\operatorname{cancel}}
\renewcommand{\deg}{\operatorname{deg}}

\newcommand{\FT}{\operatorname{FT}}
\newcommand{\Proj}{\operatorname{Proj}}
\newcommand{\Spec}{\operatorname{Spec}}

\newcommand{\dropp}{\operatorname{drop}}

\newcommand{\FLoc}{\operatorname{FLoc}}
\newcommand{\NC}{\operatorname{NC}}

{\theoremstyle{definition}
\newtheorem{defi}{Definition}[section]}
{\theoremstyle{remark}
\newtheorem{nota}[defi]{Remark}}
\newtheorem{teo}[defi]{Theorem}

\newtheorem{prop}[defi]{Proposition}
\newtheorem{lema}[defi]{Lemma}
\newtheorem{coro}[defi]{Corollary}

{\theoremstyle{remark}
}

\begin{document}

\title{Dwork families and $\cD$-modules}
\author{Alberto Castaño Domínguez}
\thanks{The author is partially supported by FQM218, P12-FQM-2696, MTM2013-46231-P, ERDF, ANR-13-IS01-0001-01/02 and DFG grants HE 2287/4-1 and SE 1114/5-1.}
\email{alcas@hrz.tu-chemnitz.de}
\address{Fakültat für Mathematik, Technische Universität Chemnitz. 09017 Chemnitz (Germany)}

\begin{abstract}
A Dwork family is a one-parameter monomial deformation of a Fermat hypersurface. In this paper we compute algebraically the invariant part of its Gauss-Manin cohomology under the action of certain subgroup of automorphisms. To achieve that goal we use the algebraic theory of $\cD$-modules, especially one-dimensional hypergeometric ones.
\end{abstract}

\maketitle

\section{Introduction}

Fix $\kk$ to be an algebraically closed field of characteristic zero. Let $n$ be a positive integer and let $w=(w_0,\ldots,w_n)\in\ZZ_{>0}^{n+1}$ be an $(n+1)$-tuple of positive integers such that $\gcd(w_0,\ldots,w_n)=1$. We will denote by $d_n$ the sum $\sum w_i\geq n+1$. Let us consider the family, parameterized by $\lambda\in\AA^1=\Spec\left(\kk[\lambda]\right)$, of projective hypersurfaces of $\PP^n=\Proj\left(\kk[x_0,\ldots,x_n]\right)$ given by:
$$x_0^{d_n}+\ldots+x_n^{d_n}-\lambda x_0^{w_0}\cdot\ldots\cdot x_n^{w_n}=0\subset\PP^n\times\AA^1.$$
This family is an example of a Dwork family, consisting of the deformation of a Fermat hypersurface by an arbitrary monomial in every variable. We will denote it by $\cXn$. By $p_n$ we will mean the restriction of the second canonical projection $\PP^n\times\AA^1\ra\AA^1$ to $\cXn$. Any other family of the form
$$a_0x_0^{d_n}+\ldots+a_nx_n^{d_n}-b\lambda x_0^{w_0}\cdot\ldots\cdot x_n^{w_n}=0\subset\PP^n\times\AA^1,$$
which a priori may appear more general, is actually isomorphic to one of the previous by the homography $x_i\mapsto\eta_ix_i$, $\lambda\mapsto b\prod\eta_i^{-w_i}\lambda$, where $\eta_i$ is a $d_n$-th root of $a_i$.

\subsection{A bit of history}

Dwork families were introduced by Dwork, especially the case $\cX_{w,\underline{1}}$, as a very interesting tool to understand the behaviour of the zeta function of a hypersurface defined over a finite field under a deformation (cf., for example, \cite{DwEstoc}). Until the nineties, the most important contributions to the knowledge of this families were due to Katz, in \cite{Ka70, KaSommes}, where respectively, he studied the differential equation satisfied by the class of the form generating the middle Gauss-Manin cohomology and showed the very interesting relation between that cohomology in the $\ell$-adic context and the $L$-function of generalized Kloosterman sums. For more information about the origins of the work with Dwork families, see the historical introduction of \cite{KaAnother}.

Since that discovery, very little was done with Dwork families until the introduction of mirror symmetry. Some of such families were a nice example of Calabi-Yau manifolds and turned out to be of interest to physicists working in that field (cf. \cite{CDR}). Because of this new interest on them, many author rediscovered Dwork families as a good tool to deal with other problems. We can cite, for instance, \cite{DouSab}, studying the quantum cohomology of weighted $n$-dimensional projective spaces or \cite{HST}, tackling the Sato-Tate conjecture in generality, among other results.

The problem of calculating some part of the Gauss-Manin cohomologies of the Dwork family $\cXn$ or the whole of it has been also addressed. We can classify the diverse works dealing with such problem by its arithmetical setting and its way of attacking the problem. From the $\ell$-adic point of view, in \cite{KaAnother} all the cohomologies of $\cXn$ are computed by using all the power of the étale machinery. The same techniques are used in \cite{RW}, as a way to find the moment zeta functions for the original Dwork family and its quotient under the action of certain group of automorphisms. These works are quite complete and make the most of the $\ell$-adic approach.

In the $p$-adic setting, a significative part of the work is carried out referring to Dwork's classical methods of $p$-adic analysis. We can cite \cite{Kl}, giving the expression of the matrix of the associated integrable Gauss-Manin connection using a mixture of rigid and Monsky-Washnitzer cohomology theories and Dwork's techniques, or \cite{Yu}, studying the variation of the unit root of the original Dwork family, also with methods going back to Dwork with a touch of crystalline cohomology.

Over the complex numbers, we have a different approach, computational but still after Dwork and Katz, in \cite{Sal}. Also Yu and Katz treat this setting in their respective papers \cite{Yu} and \cite{KaAnother} above referred to. The former finds an horizontal section of the Gauss-Manin connection, whereas the latter expresses the same invariant part of the Gauss-Manin cohomologies issued in this paper in terms of hypergeometric $\cD$-modules using analytical transcendental methods.

In this paper we deal with that last problem, but in a purely algebraic way by using the power of the theory of $\cD$-modules over any algebraically closed field of characteristic zero. This work can be seen as a first step in two directions. Namely, one is finding, as Katz and Kloosterman do in their papers but in an algebraic way as in the present paper, the expression of all the Gauss-Manin cohomologies of a Dwork family, not only some invariant part. The second is the extension of this work to the $p$-adic world, with the help of a good theory of $p$-adic cohomology (cf. \cite{AbeCaro} or \cite{MebConstr} to see the two existing approaches).

Before ending this introduction, let us introduce a couple of conventions. An algebraic variety, or just variety, will mean for us a separated, finite type, equidimensional scheme over $\kk$, reducible or not. We will denote by $\pi_i$ the $i$-th canonical projection from a product, and by $\pi_X$ the projection to a point from any variety $X$.

\subsection{Some notions on $\cD$-modules}

Before going on, we will digress a little and recall some notions from algebraic $\cD$-module theory that will be useful to state properly the main results of this paper (the references used here in this regard are \cite{HTT,KaESDE,Meb6Oper}). Associated with a smooth algebraic variety $X$, let us denote by $\cO_X$ and $\cD_X$ its structure sheaf and the sheaf of differential operators on it, respectively. We will denote by $\D(\cD_X)$ the derived category of bounded complexes of (left) $\cD_X$-modules. We can also define the derived categories of bounded complexes of $\cD_X$-modules with coherent, holonomic and regular holonomic cohomologies, each of them being a full triangulated subcategory of the precedent. Whenever we talk about a complex of $\cD_X$-modules, we will understand it as an object of one of the derived categories above, which will be clear from the context.

Associated with a morphism of smooth varieties $f:X\ra Y$, we will denote the usual direct and inverse image functors by $f_+$ and $f^+$, and the extraordinary direct and inverse images by $f_!:=\mathbb{D}_Yf_+\mathbb{D}_X$ and $f^!:=\mathbb{D}_Yf^+\mathbb{D}_X$, respectively, $\mathbb{D}_Z$ being the duality functor on a variety $Z$. Regarding tensor products, we will write $\otimes_{\cO_X}^{\bL}$ and $\boxtimes$ for the interior and exterior ones, respectively. We will also use extensively the intermediate extension, a third image functor under some special conditions. For any open subvariety $U$ of $X$ consider the inclusion $j:U\ra X$. If $j$ is affine, the intermediate (or middle) extension of a holonomic $\cD$-module $\cM$ is the image of the canonical morphism $j_!\cM\ra j_+\cM$; it is denoted by $j_{!+}\cM$ (see more details in \cite[\S 3.4]{HTT}).

Recall that every holonomic $\cD$-module is both Noetherian and Artinian and then of finite length, so all of them admit a composition series. This allows us to define the semisimplification of a holonomic $\cD$-module as the direct sum of all of its composition factors.

Another relevant functor in this work is Fourier (or Fourier-Laplace) transformation. It is a very important functor in $\cD$-module theory, of which we will only treat here its absolute one-dimensional version. It can be defined as an endofunctor of $\D(\cD_{\AA^1})$ given by $\FT=\pi_{2,+}\left(\pi_1^+\bullet \otimes_{\cO_{\AA^2}}^{\bL}\cL\right)$, where $\cL$ is the holonomic $\cD_{\AA^2}$-module $\cD_{\AA^2}/(\dd_x-y,\dd_y-x)$.

Fourier-Laplace transformation preserves coherence over $\cD_{\AA^1}$ and holonomy but not regularity; it is an equivalence of categories when defined on the associated full triangulated subcategories of $\D(\cD_{\AA^1})$ with the first two properties.

We will say that a $\cD$-module is constant if it is formed as a successive extension of structure sheaves. Usually constant modules are not quite interesting, and it would be good to have some way of isolating the ``nonconstant'' part of a $\cD$-module. Here it is a way to do so in the concrete one-dimensional setting that is of interest to us.

\begin{defi}\label{nonconstant}
Let $\cM$ be a holonomic $\cD_{\AA^1}$-module, and let $j$ be the canonical inclusion of $\gm$ into $\AA^1$. We define its nonconstant part as
$$\NC\cM=\FT^{-1}j_{!+}j^+\FT\cM.$$
In the same vein, given a holonomic $\cD_{\gm}$-module $\cM$, we will define its nonconstant part as
$$\NC_0\cM=j^+\NC j_{!+}\cM.$$
\end{defi}

\begin{nota}
The functors $\NC$ and $\NC_0$ are exact in the category of holonomic $\cD$-modules, respectively over $\AA^1$ and $\gm$, and preserve irreducible objects. An irreducible holonomic $\cD$-module has zero nonconstant part if and only if it is $\cO$. Therefore, for any holonomic $\cD$-module $\cM$ its nonconstant composition factors and those of $\NC\cM$ are the same.
\end{nota}

Recall also that, if $X$ is of dimension $n$ and $\cM$ is a complex of holonomic $\cD_X$-modules, the Euler-Poincaré characteristic of $\cM$ is
$$\chi(\cM)=(-1)^n\sum_{k}(-1)^k\dim\cH^k\pi_{X,+}\cM.$$

Hypergeometric $\cD$-modules are a crucial ingredient of this work. Even though we will explain them in detail in section \ref{hyper}, let us just define them here.

\begin{defi}\label{hipergeometricos}
Let $(n,m)$ be a couple of nonnegative integers, and let $\alpha_1,\ldots,\alpha_n$ and $\beta_1,\ldots,\beta_m$, respectively, $n$ and $m$ elements of $\kk$ and $\gamma\in\kk^*$. The hypergeometric $\cD$-module associated with $\gamma$, the $\alpha_i$ and the $\beta_j$ is defined as the quotient of $\cD_{\gm}$ with the ideal generated by the so-called hypergeometric operator
$$\gamma\prod_{i=1}^n(D-\alpha_i)- \lambda\prod_{j=1}^m(D-\beta_j).$$
We will denote it by $\cH_\gamma(\alpha_1,\ldots,\alpha_n;\beta_1,\ldots,\beta_m)$, or in an abridged way, $\cH_{\gamma}(\alpha_i;\beta_j)$.
\end{defi}

\subsection{Contents of the paper}\label{contents}

Denote by $(\alpha_1,\ldots,\alpha_n)$ and $(\beta_1,\ldots,\beta_m)$ two unordered $n$- and $m$-tuples (multisets), respectively, of elements from $\kk$. We will define their cancelation, denoted by $\cancel(\alpha_1,\ldots,\alpha_n;\beta_1,\ldots,\beta_m)$, as the result of eliminating from both tuples the elements that they share modulo $\ZZ$, obtaining shorter disjoint lists (cf. \cite[6.3.10]{KaESDE}). In other words, assume that, up to a reordering on the $\alpha_i$ and the $\beta_j$, there exists an integer $0\leq r\leq\min(n,m)$ such that
$$\begin{array}{rcll}
\alpha_k&\equiv&\beta_k \mod \ZZ& \text{ for }k\leq r,\\
\alpha_k&\not\equiv&\beta_k \mod \ZZ&\text{ for } k>r.
\end{array}$$
Then, under this assumption,
$$\cancel(\alpha_1,\ldots,\alpha_n;\beta_1,\ldots,\beta_m)= (\alpha_{r+1},\ldots,\alpha_n;\beta_{r+1},\ldots,\beta_m).$$
If $k=\min(n,m)$ at least one of the tuples would be empty. Note that the subtractions remove only an element of the first tuple for each similar element of the second one, and thus the difference between the lengths of the resulting tuples is still the same.

There exists an important subgroup of the group of automorphisms of $\cXn$, being the one which will provide us our main object of study. Let
$$\mu_{d_n}^0=\left\{(\zeta_0,\ldots,\zeta_n)\in\mu_{d_n}^{n+1}\bigg| \prod_{i=0}^n\zeta_i^{w_i}=1\right\},$$
and let $G=\mu_{d_n}^0/\Delta$, the quotient of $\mu_{d_n}^0$ by the diagonal subgroup, acting linearly over $\cXn$ by componentwise multiplication on the projective variables.

Let $X$ be a quasi-projective variety and $U$ an open subvariety of the affine line. Given a smooth family of hypersurfaces $\cX\subset X\times U$, together with the restriction of the second canonical projection $p:\cX\ra U$, we will define the Gauss-Manin complex of $\cX$ relative to $U$ as the direct image $p_+\cO_{\cX}$, understood as a complex of regular holonomic $\cD_U$-modules. The cohomologies of such complex, specially the middle one, are the Gauss-Manin cohomologies of $\cX$ (relative to $U$).

Our goal is to calculate the invariant part under the action of $G$ of the Gauss-Manin complex of $\cXn$ relative to the parameter $\lambda$, or in other words, $\left(p_{n,+}\cO_{\cXn}\right)^G$. This must be understood as follows.

$G$ acts linearly over $\cXn$, and in fact, over $\PP^n\times\AA^1$. Then for any $G$-equivariant $\cD_{\cXn}$- or $\cD_{\PP^n\times\AA^1}$-module (cf. \cite[Def. 3.1.3]{KashEquiv}, noting that quasi-$G$-equivariance and $G$-equivariance coincide due to the finiteness of $G$) we have an action of the Lie algebra associated with $G$, which, by the latter being finite and abelian, is the commutative group algebra $\kk[G]$. In our case, $\cO_{\cXn}$ is a $G$-equivariant $\cD_{\cXn}$-module, and since $p_n$ is a $G$-equivariant morphism by definition, the direct image $p_{n,+}$ induces a $G$-equivariant structure on $p_{n,+}\cO_{\cXn}$ (cf. [\textit{op. cit.}, p. 169]). In this sense, whenever we talk about the invariant part of a $\cD$-module $\cM$, we will understand its image by the functor $\Hom_{\kk[G]}(\kk,\bullet)$.

In fact, if we focus on the nonconstant part of the cohomologies of $\left(p_{n,+}\cO_{\cXn}\right)^G$, following Definition \ref{nonconstant}, we can state our first main result:

\begin{teo}\label{interesante}
Let $j$ be the canonical inclusion from $\gm$ to $\AA^1$. Let $\cH_n$ be the irreducible hypergeometric $\cD$-module
$$\cH_{\gamma_n} \left(\cancel\left(\frac{1}{w_0},\ldots,\frac{w_0}{w_0},\ldots, \frac{1}{w_n},\ldots,\frac{w_n}{w_n}; \frac{1}{d_n},\ldots,\frac{d_n}{d_n}\right)\right).$$
Then only the zeroth cohomology of $\left(p_{n,+}\cO_{\cXn}\right)^G$ has a nonconstant part, which is $j_{!+}\iota_n^+\cH_n$, with $\iota_n:\gm\ra\gm$ being the map given by $z\mapsto z^{-d_n}$.
\end{teo}

As we said in the introduction, this result was obtained by Katz and Kloosterman in different contexts (cf. \cite[Thm. 5.3, Cor. 8.5]{KaAnother} and \cite[Prop. 5.3]{Kl}, respectively). In that sense, Theorem \ref{interesante} is the analogous statement in the algebraic case. However, we can go further and be more precise than them, as the following theorem shows.

For each $n$, let us denote by $C_n$ the following set:
$$C_n=\left\{k/d_n\,:\,k/d_n=j/w_i\text{, for some }i=0,\ldots,n, j=1,\ldots,w_i-1,k=1,\ldots,d_n-1\right\}.$$
Note that this set will be empty if and only if $d_n$ is prime to each $w_i$.

Denote by $\bar{K}_n$ the Gauss-Manin complex of the smooth affine part of the quotient $\cXn/G$ (see Section \ref{problema} for more details). We will prove first Theorem \ref{teorema1} stating that $\left(p_{n,+}\cO_{\cXn}\right)^G$ and $\bar{K}_n$ differ only in some constant part, so we can focus on calculate $\bar{K}_n$. We can give more information about it in our next main theorem, in such a way that Theorem \ref{interesante} can be seen as a corollary of it:

\begin{teo}\label{teorema2}
The restriction $j^+\bar{K}_n$ is the inverse image under $\iota_n$ of another complex $K_n\in\Drh(\cD_{\gm})$, for which the following holds: $\cH^iK_n=0$ if $i\notin\{-(n-1),\ldots,0\}$, $\cH^iK_n\cong\cO_{\gm}^{\binom{n}{i+n-1}}$ as long as $-(n-1)\leq i\leq-1$, and in degree zero we have the exact sequence
$$0\lra\cG_n\lra \cH^{0}K_n\lra\cO_{\gm}^n\lra0,$$
$\cG_n$ lying in an exact sequence of the form
$$0\lra\cH_n\lra\cG_n\lra\bigoplus_{\alpha\in C_n}\cK_{\alpha}\lra0,$$
where $\cK_\alpha$ stands for the Kummer $\cD$-module $\cD_{\gm}/(D-\alpha)$.
\end{teo}

The statement of the theorem could seem a vague definition of a new complex $K_n$, but its point is giving in detail the expression of $\bar{K}_n$ in a simple way. Nevertheless, we will define $K_n$ properly in section \ref{problema}.

Note that \cite[Thm. 2.1]{RW} is the $\ell$-adic analogue of Theorem \ref{teorema2} for the original Dwork family, without mentioning the second exact sequence. In this sense our statement about the invariant part of the Gauss-Manin complex of the Dwork family is finer than any other result in any other context.

The paper is structured as follows. Section \ref{hyper} is a compilation of basic results on hypergeometric $\cD$-modules that will be of use in the following three sections. In section \ref{problema} we introduce properly our problem and start to tackle our main results. Next we exploit in section \ref{induccion} the power of the inductive method used to prove most of Theorem \ref{teorema2}, whereas in section \ref{fourier} we finish its proof by finding the exponents of $\cG_n$ at the origin and infinity.

\textbf{Acknowledgements.} This work addresses a problem introduced by Steven Sperber to the author, to whom he is deeply grateful. The results of this paper conform an improvement of the main part of the doctoral thesis of the author, advised by Luis Narv\'{a}ez Macarro and Antonio Rojas Le\'{o}n. The author wants to thank the latter for having the idea of the inductive framework, and both of them for their support and encouragement. He is also grateful to Claude Sabbah for suggesting him the use of Fourier transformation and pointing him out the reference \cite{DouSab}, and an anonymous referee for the useful comments given to improve the paper. He wants to thank as well his son, Alberto Luna Castaño, born during the early writing of this paper, for making him learning many new things and improving his productivity. This work is dedicated to him.

\section{Hypergeometric $\cD$-modules}\label{hyper}

In this section we introduce hypergeometric $\cD$-modules, studying their basic properties and their parameters, ending by characterizing such a $\cD$-module by them. Almost any result stated here is inspired by or directly equal to some other of \cite[\S3]{KaESDE}. In fact, apart from Katz's reference, we have not found in the literature an algebraic approach to hypergeometric $\cD$-modules as his.

Throughout the text, we will remain with one-dimensional $\cD$-modules. From now on, we will denote by $D_\lambda$ the product $\lambda\dd_\lambda$, omitting the variable as long as it is clear from the context.

\begin{prop}\label{charcero}
Recall that a Kummer $\cD$-module is the quotient $\cK_\alpha=\cD_{\gm}/(D-\alpha)$, for any $\alpha\in\kk$. Let $\cM$ be a holonomic $\cD_{\gm}$-module. Then its Euler-Poincaré characteristic is zero if and only if its composition factors are Kummer $\cD$-modules.
\end{prop}
\begin{proof}
The Euler-Poincaré characteristic of a holonomic $\cD_{\gm}$-module is additive and never positive, so we just need to show the equivalence when $\cM$ is irreducible.

If $\cM=\cK_\alpha$, since it has no singularities at $\gm$ we can apply Deligne's formula at \cite[2.9.8.2]{KaESDE}. In this case, $\cM$ has a regular singularity both at zero and infinity, so $\chi(\cM)=0$.

Suppose then that $\chi(\cM)=0$. By [\textit{op. cit.}, Cor. 2.9.6.1], we have that $\cM$, being irreducible, coincides with $j_{!+}j^+\cM$, for any open subvariety $U\stackrel{j}{\hra} \gm$ of $\gm$ in which $\cM$ is a module with integrable connection. Then we can apply now Deligne and Gabber's formula [\textit{op. cit.}, Thm. 2.9.9]. Since $\chi(\cO_{\gm})=0$, we have necessarily that all the singularities of $\cM$ are regular, and, following Katz's notation, $\dropp_p\cM=0$ at every point of $\gm-U$. Therefore $\cM$ can be seen as a module with integrable connection on $\gm$, so there exists a square matrix $A$ of elements of $\kk\left[\lambda^{\pm}\right]$ of order $r$ such that $\cM\cong\cD_{\gm}^r/(D-A)$. Now we can apply to this connection the theorem of the cyclic vector \cite[Thm. III.4.2]{DGS} to obtain that $\cM$ is isomorphic to $\cD_{\gm}/(P(\lambda,D))$, with $P$ being a nonconstant polynomial of $\kk\left[y^\pm,t\right]$. (Note that although the proof of the theorem of the cyclic vector requires the connection to be defined over a function field of the form $\kk(\lambda)$, it only actually needs that $\lambda$ is invertible, as in our case, providing a global proof in $\gm$.) Now $\cM$ has a regular singularity at zero, so the degree in $\lambda$ of the coefficients of $P$ cannot be negative (cf. \cite[1.1.2]{Deligne}). Taking $\mu=\lambda^{-1}$, we have $D_\mu=-D_\lambda$, so by the same argument, the order at $\mu$ of the coefficients of $P$ must be zero at least too, and thus $P$ is a polynomial only in $D$. Since $\kk$ is algebraically closed and $\cM$ is irreducible, $\deg P(D)=1$, for if it were greater we would have a composition factor of $\cM$ consisting of the quotient of $\cD_{\gm}$ by the left ideal generated by any factor of $P(D)$. In conclusion, $\cM\cong\cD_{\gm}/(D-\alpha)$.
\end{proof}

We have just characterized, up to semisimplification, the holonomic $\cD_{\gm}$-modules of Euler-Poincaré characteristic zero. Regarding characteristic $-1$, we can claim that the composition factors of any holonomic $\cD_{\gm}$-module of Euler-Poincaré characteristic $-1$ will be a finite amount of Kummer $\cD$-modules and an irreducible $\cD_{\gm}$-module of characteristic $-1$, which can be an irreducible punctual $\cD_{\gm}$-module supported at any point of $\gm$, also called delta $\cD$-module. Irreducible hypergeometric $\cD$-modules of positive rank, as we will see, are the irreducible nonpunctual holonomic $\cD_{\gm}$-modules of characteristic $-1$. Let us return to them.

\begin{nota}
The special type $(n,m)=(0,0)$ in Definition \ref{hipergeometricos} corresponds to delta $\cD$-modules on $\gm$, since
$$\cH_\gamma\left(\emptyset;\emptyset\right)=\cD_{\gm}/(\gamma-\lambda).$$
On the other hand, is easy to check that $\cH_\gamma(\alpha_i;\beta_j) \otimes_{\cO_{\gm}}\cK_\eta\cong \cH_\gamma(\alpha_i+\eta;\beta_j+\eta)$. In this paper we will only focus on the case $n=m$, so that $\cH_\gamma(\alpha_i;\beta_j)$ has only regular singularities, at the origin, infinity and $\gamma$, where the Jordan decomposition of its local monodromy (when $\kk=\CC$) is a pseudoreflection, that is, the space of formal meromorphic solutions $\Hom(\cH,\kk((x_{\gamma})))$ is $(n-1)$-dimensional (cf. \cite[Lem. 2.9.7]{KaESDE}.
\end{nota}

Any hypergeometric $\cD$-module is of Euler-Poincaré characteristic $-1$ (\hspace{-.5pt}\cite[Lem. 2.9.13]{KaESDE}) and its behaviour at $\PP^1$ can be fully understood from its parameters, as we are going to show in the following results.

\begin{prop}[Irreducibility]
\emph{(cf. \cite[Props. 2.11.9, 3.2]{KaESDE})} Let $\cH:=\cH_\gamma(\alpha_i;\beta_j)$ be a hypergeometric $\cD$-module. It is irreducible if and only if $\alpha_i-\beta_j$ is not an integer for any pair $(i,j)$ of indexes.
\end{prop}

The next proposition deals with the exponents of hypergeometric $\cD$-modules. As we will shortly see, they are an important property in order to characterize them. For more information on the exponents of a $\cD$-module in dimension one, see \cite[\S2]{CaKoszul}.

\begin{prop}[Exponents]\label{exponentes}
\emph{(cf. \cite[Cor. 3.2.2]{KaESDE})} Let $\cH:=\cH_\gamma(\alpha_i;\beta_j)$ be an irreducible hypergeometric $\cD$-module of type $(n,m)$, and fix a fundamental domain $I$ of $\kk/\ZZ$. Then,
\begin{enumerate}[i)]
\item The Jordan decomposition of the regular part of $\cH$ at the origin is $$\left(\cH\otimes\kk((\lambda))\right)_{\emph{slope=0}}\cong\bigoplus_{\alpha\in I}\kk((\lambda))[D_\lambda]/(D_\lambda-\alpha)^{n_\alpha},$$
    where $n_\alpha$ is the amount of $\alpha_i$ congruent to $\alpha$ modulo $\ZZ$.
\item The Jordan decomposition of the regular part of $\cH$ at infinity is
    $$\left(\cH\otimes\kk((\mu))\right)_{\emph{slope=0}}\cong\bigoplus_{\beta\in I}\kk((\mu))[D_\mu]/(D_\mu-\beta)^{n_\beta},$$
    where $\mu=1/\lambda$ and $n_\beta$ is the number of $\beta_i$ congruent to $\beta$ modulo $\ZZ$.
\end{enumerate}
\end{prop}

\begin{prop}[Isomorphism class]\label{isohyp}
\emph{(\hspace{-.5pt}\cite[Lem. 3.3]{KaESDE})} Let $\cH=\cH_\gamma(\alpha_i;\beta_j)$ be a hypergeometric $\cD$-module of type $(n,m)$. Its isomorphism class as $\cD_{\gm}$-module determines $n$ and $m$, the set of all of the $\alpha_i$ and $\beta_j$ modulo $\ZZ$, and if either $\cH$ is irreducible or $n=m$, the point $\gamma$.
\end{prop}

As a matter of fact, this proposition shows that in the regular or the irreducible case, all the parameters of a hypergeometric $\cD$-module are intrinsic to it. Namely, assuming $n\geq m$, $n$ is the generic rank of $\cH$, $m$ the rank of its slope zero part at infinity, the $\alpha_i$ and the $\beta_j$ are the exponents at the origin and infinity, respectively, and $\gamma$ is the other regular singularity of $\cH$ in the case it is regular. The argument in the irreducible irregular case is a bit more complicated, but we will not need it in the following.

\begin{prop}\label{hyperLS}
\emph{(cf. \cite[Thm. 3.7.1]{KaESDE}, \cite[Thm. 1]{LS})} Let $\cM$ be a nonpunctual irreducible holonomic $\cD_{\gm}$-module of Euler-Poincaré characteristic $-1$. Then $\cM$ is a hypergeometric $\cD$-module.
\end{prop}

%Apart from what we have seen, more facts about hypergeometric $\cD$-modules are known. Namely, they are rigid (cf. \cite[Thms. 3.5.4, 3.7.3]{KaESDE}), and in the regular irreducible case, so that the restriction of such a $\cD$-module to $\gm$ minus the third regular singular point underlies a complex polarizable variation of Hodge structures, we can calculate their Hodge numbers thanks to \cite[Thm. 1]{Fe}.

\section{Setting the problem}\label{problema}

Once we have introduced the objects to study and the tools to work with, let us keep on and start to work on the problem addressed in the paper.

\begin{prop}
$\cXn$ is a smooth quasi-projective variety. Write
$$\gamma_n=(d_n)^{-d_n}\prod_{i=0}^nw_i^{w_i}.$$
Then the morphism $p_n$ is smooth over the open subvariety $U_n=\left\{\lambda\in\AA^1\,|\,\gamma_n\lambda^{d_n}\neq1\right\}$.
\end{prop}
\begin{proof}
Let us use the multi-index notation, for the sake of simplicity, in such a way that $\xx^a$ will represent the monomial $x_0^{a_0}\cdot\ldots\cdot x_n^{a_n}$ for any $(n+1)$-uple $a$ of positive integers. The partial derivatives of $x_0^{d_n}+\ldots+x_n^{d_n}-\lambda\xx^w$ with respect to the $x_i$ and $\lambda$ are
$$\delta_i:=d_nx_i^{d_n-1}-\lambda w_i\xx^{w-e_i}\,\text{ and }\delta_\lambda:=-\xx^w,$$
respectively. If $\delta_\lambda=0$, then some $x_i$ must vanish, and if $\delta_i=0$ for every $i$, all of the $x_i$ will be zero, which is impossible. Therefore, $\cXn$ is smooth. Since at the singular points of the fibers of $p_n$ we have that $x_i\neq0$ for every $i$, we can multiply by them the partial derivatives $\delta_i$, obtaining that $d_nx_i^{d_n}=-w_i\lambda\xx^w$ for every $i$, so $w_0x_i^{d_n}=w_ix_0^{d_n}$. But then, substituting the $x_i^{d_n}$ in the equation of $\cXn$, $(d_n/w_0)x_0^{d_n}=\lambda\xx^w$. Taking $d_n$-th powers at each side, we get that $d_n^{d_n}=\lambda^{d_n}\prod_iw_i^{w_i}$, so if we do not have that equality, we will find ourselves at a nonsingular fiber.
\end{proof}

Recall that we had an action of the group $G$ on the Dwork family $\cXn$. Since $G$ is a finite group, $\cXn/G$ is another quasi-projective variety, isomorphic to
$$\left\{\begin{array}{rcl}
x_0+\ldots+x_n&=&\lambda x_{n+1}\\
x_0^{w_0}\cdot\ldots\cdot x_n^{w_n}&=&x_{n+1}^{d_n}\end{array}\right.\subset\PP^{n+1}\times\AA^1.$$
Substituting $x_0$ in the second equation for its value given by the first one and taking $x_{n+1}=1$, we find that $\cXn/G$ is the projective closure of
$$\cYn :x_1^{w_1}\cdot\ldots\cdot x_n^{w_n}(\lambda-x_1-\ldots-x_n)^{w_0}=1\subset\gm^n\times\AA^1= \Spec\left(\kk\left[x_1^\pm,\ldots,x_n^\pm,\lambda\right]\right).$$
In this sense we will write $\bar{\cY}_{n,w}$ for $\cXn/G$.

Let $\cZ_{n,w}\subset\AA^n\times\AA^1=\Spec\left(\kk\left[x_1,\ldots,x_n,\lambda\right]\right)$ be the variety defined by the equation
$$x_1^{w_1}\cdot\ldots\cdot x_n^{w_n}\cdot(1-x_1-\ldots-x_n)^{w_0}=\lambda.$$
We will omit the $(n+1)$-tuple $w$ from $\cZ_{n,w}$ when it is clear from the context.

Denote by $'p_n$ and $''p_n$ the projections from $\cYn$ and $\cZ_{n,w}$, respectively, to $\AA^1$. Then we can form the following cartesian diagram:
\begin{equation}\label{cartesiano YZ}
\xymatrix{\ar @{} [dr] |{\Box} \cYn- {}'p_n^{-1}(0) \ar[d]_{'p_n} \ar[r]^{\tilde{\alpha}_n} & \cZ_{n,w}- {}''p_n^{-1}(0) \ar[d]^{''p_n}\\
 \gm \ar[r]^{\iota_n} & \gm},
\end{equation}
where $\tilde{\alpha}_n(\underline{x},\lambda)=\left((x_1/\lambda,\ldots, x_n/\lambda), \lambda^{-d_n}\right)$ (note that it is the restriction of an endomorphism of $\gm^{n+1}$).

\begin{prop}
The families $\cYn$ and $\cZ_{n,w}$ are smooth, but not their projective closures. The projections $'p_n$ and $''p_n$ are smooth, respectively, over the open subvarieties $U_n$ and $\gm-\{\gamma_n\}$.
\end{prop}
\begin{proof}
$\cYn$ and $\cZ_{n,w}$ are smooth for being, respectively, a $n$-dimensional torus and a graph. However, their projective closures will have singularities at their sections at infinity independently of $\lambda$, because both of them are the cartesian product of the same arrangement of hyperplanes with $\AA^1$.

Regarding the fibers of $'p_n$, since the section at infinity of $\bar{\cY}_{n,w}$ is independent of $\lambda$, the singular fibers of $\cYn$ will be over the same points of $\AA^1$ as those of $\bar{\cY}_{n,w}$. Now the quotient map $\pi_G:\cXn\ra\bar{\cY}_{n,w}$ is $G$-equivariant by definition, as well as $p_n$. Then the singular locus of $'p_n:\bar{\cY}_{n,w}\ra\AA^1$ is the same as that of $p_n:\cXn\ra\AA^1$, which is $U_n$.

With respect to the fibers of $\cZ_{n,w}$, note that $\tilde{\alpha}_n$ is an étale morphism out of its section with equation $\lambda=0$, where it is ramified. Therefore, $\cZ_{n,w}$ will have nonsingular fibers on the image by $\iota_n$ of $U_n$ except the origin, that is to say, $\gm-\{\gamma_n\}$.
\end{proof}

Recall that we wanted to compute the invariant part under the action of $G$ of the Gauss-Manin complex of $\cXn$ relative to the parameter $\lambda$, that is, $\left(p_{n,+}\cO_{\cXn}\right)^G$. Remember as well that what is actually more interesting to us is the nonconstant part of its cohomologies. We can actually restrict ourselves to an affine context in order to find it:

\begin{teo}\label{teorema1}
Let $\bar{K}_n={}'p_{n,+}\cO_{\cYn}$. There exists a canonical morphism between the complexes of $\cD_{\AA^1}$-modules $\left(p_{n,+}\cO_{\cXn}\right)^G\lra\bar{K}_n$ such that the cohomologies of its cone are direct sums of copies of the structure sheaf $\cO_{\AA^1}$.
\end{teo}
\begin{proof}
Since $\cXn$ is smooth, $\bR\Gamma_{\left[\cXn\right]}\cO_{\PP^n\times\AA^1}\cong\iota_+\cO_{\cXn}[-1]$, where $\iota$ denotes the inclusion $\cXn\hra\PP^n\times\AA^1$ (cf. \cite[Prop. 1.7.1]{HTT}). The action of $G$ on $\cXn$ can be easily extended to $\PP^n\times\AA^1$, in such a way that $\iota$ is a $G$-equivariant morphism. Then, as discussed in subsection \ref{contents}, we can claim by \cite[p. 169]{KashEquiv} that $\bR\Gamma_{\left[\cXn\right]}\cO_{\PP^n\times\AA^1}$ is a $G$-equivariant $\cD_{\PP^n\times\AA^1}$-module.

Let us see now $\bar{\cY}_{n,w}$ as a quasi-projective variety in $\PP^n\times\AA^1$ and call $\cM:=\bR\Gamma_{\left[\bar{\cY}_{n,w}\right]}\cO_{\PP^n\times\AA^1}$. Let $\mathcal{J}_{\cX}$ and $\mathcal{J}_{\cY}$ be the ideals of definition of $\cXn$ and $\cXn/G$, respectively. The invariant part under the action of $G$ of the rings $\cO_{\PP^n\times\AA^1}/\mathcal{J}_\cX^k$ are, by construction, $\cO_{\PP^n\times\AA^1}/\mathcal{J}_\cY^k$, for every $k\geq0$ (cf. \cite[p. 127]{Harris}). Since we are working with a finite abelian group and sheaves of $\kk$-vector spaces, we can claim thanks to Maschke's theorem that the functor $\bullet^G=\Hom_{\kk[G]}(\kk,\bullet)$ is exact.

Furthermore, $G$ is finite, and thus isomorphic to the product of some cyclic groups. Then the invariant part of a sheaf of $\kk$-vector spaces is the kernel of the product of the linear maps $\varphi_{a_i}-\id$, the $a_i$ and $\varphi_{a_i}$ being the generators of $G$ and their associated actions on the sheaf. Since $\bullet^G$ is a kernel and an exact functor, it commutes with derived functors of left exact ones. In particular, so it does with $\bR\cH om_{\cO_{\PP^n\times\AA^1}}\left(\bullet,\cO_{\PP^n\times\AA^1}\right)$, and then the invariant part of $\bR\Gamma_{\left[\cXn\right]}\cO_{\PP^n\times\AA^1}$ under the action of $G$ must be $\cM$.

By the same arguments as in the beginning of the proof, $\pi_{2,+}\bR\Gamma_{\left[\cXn\right]}\cO_{\PP^n\times\AA^1}[1]\cong p_{n,+}\cO_{\cXn}$. Let us prove now that taking invariants and direct image by $\pi_{2}$ commute, so that $\left(p_{n,+}\cO_{\cXn}\right)^G=\pi_{2,+}\cM[1]$.

The morphism $\pi_2$ is a projection, so the functor $\pi_{2,+}$ is the image by $\bR\pi_{2,*}$ of the relative de Rham complex $\DR_{\pi_{2}}$ shifted $n$ degrees to the left. By the same reasons as in the previous paragraph, $\bR\pi_{2,*}$ commutes with $\bullet^G$. The relative de Rham complex is a complex of sheaves of $\kk$-vector spaces whose objects are $\cN\otimes_{\cO_{\PP^n\times\AA^1}}\Omega_{\PP^n\times\AA^1/\AA^1}^i$ for some $\cD_{\PP^n\times\AA^1}$-module $\cN$. The connecting morphisms are $\kk$-linear, and then $G$-equivariant. Since locally the differential modules are isomorphic to a direct sum of copies of $\cO_{\PP^n\times\AA^1}$, the objects of the relative de Rham complex will be locally isomorphic to a direct sum of copies of $\cN$. Consequently, $\bullet^G$ and $DR_{\pi_{2}}$ will commute as well and then, $\left(p_{n,+}\cO_{\cXn}\right)^G=\pi_{2,+}\cM[1]$.

Let now $\cYn^{\infty}$ be the intersection of the hyperplane at infinity with $\bar{\cY}_{n,w}$ and denote by $i:\PP^{n-1}\times\AA^1\ra\PP^n\times\AA^1$ and $j:\AA^n\times\AA^1\ra\PP^n\times\AA^1$ the canonical immersions. We have the associated excision distinguished triangle
$$\pi_{2,+}\cM\lra \pi_{2,+}j^+\cM\lra \pi_{2,+}i^+\cM\lra.$$
Now, by \cite[Prop. I.6.3.1]{Meb6Oper},
$$i^+\cM\cong \bR\Gamma_{[\cYn^{\infty}]}\cO_{\PP^{n-1}\times\AA^1}\cong \pi_1^+\bR\Gamma_{\left[\bar{A}\right]}\cO_{\PP^{n-1}},$$
$\bar{A}$ being the projective arrangement of hyperplanes such that $\cYn^{\infty}$ is the product $\bar{A}\times\AA^1$. Then $\pi_{2,+}i^+\cM$ is the tensor product $\pi_{\PP^{n-1},+}\bR\Gamma_{\left[\bar{A}\right]}\cO_{\PP^{n-1}} \otimes_\kk\cO_{\AA^1}$ and it has constant cohomologies.

Finally, $j^+\cM\cong\bR\Gamma_{[\cYn]}\cO_{\AA^n\times\AA^1}$. Since $\cYn$ is smooth, we have $\pi_{2,+}j^+\cM\cong\bar{K}_n[-1]$, so in the end we have a triangle
$$\left(p_{n,+}\cO_{\cXn}\right)^G\lra\bar{K}_n\lra \pi_{2,+}i_+i^+\cM[1]\lra,$$
the last complex having direct sums of copies of the structure sheaf as cohomologies, and we are done.
\end{proof}

\begin{nota}\label{semisimple}
The morphism $p_n$ is proper and smooth and $\cO_{\cXn}$ can be endowed with the structure of a pure Hodge module. Then, by virtue of \cite[4.5.2-4.5.4]{SaiMHM}, we can affirm that $p_{n,+}\cO_{\cXn}$ can be upgraded to a semisimple complex of Hodge modules, being in such a case the direct sum of its cohomologies, which are in turn semisimple Hodge modules. The same applies to their underlying $\cD_{\AA^1}$-modules. Now note that $\left(p_{n,+}\cO_{\cXn}\right)^G$ is a direct summand of $p_{n,+}\cO_{\cXn}$, because the latter can be decomposed as the direct sum of its eigenspaces associated with the action of $G$, so in particular, the former is a semisimple complex of $\cD_{\AA^1}$-modules too.

In particular, the theorem tells us that the nonconstant part of the cohomologies of $\left(p_{n,+}\cO_{\cXn}\right)^G$ are those of $\bar{K}_n$, which coincide with the middle extension of their restrictions to $\gm$. Note indeed that due to the theorem, $\bar{K}_n$ has the same singularities at the origin as $\left(p_{n,+}\cO_{\cXn}\right)^G$, which as said in the paragraph above, is a direct summand of $p_{n,+}\cO_{\cXn}$ which has in turn no singularity at the origin, for $p_n$ is smooth there. Summing up, the nonconstant part of the cohomologies of $\bar{K}_n$ cannot have any punctual composition factor at the origin.
\end{nota}

Let us continue now with the proof of Theorem \ref{interesante}, by using the theorem above. We need some more information about the Fourier transform of $\bar{K}_n$ and we will gather it from the work \cite{DouSab} by Douai and Sabbah.

\begin{prop}\label{FTKn}
\emph{(cf. \cite[Prop. 3.2]{DouSab})} Let $n\geq1$. Then we have an isomorphism of complexes of $\cD_{\gm}$-modules
$$j^+\FT\bar{K}_n\cong[d_n]^+\cH_{d_n^{d_n}\gamma_n} \left(\frac{1}{w_0},\ldots,\frac{w_0}{w_0},\ldots,\frac{w_n}{w_n}; \emptyset\right).$$
\end{prop}

\vspace{-.5cm}

\begin{defi}
Let $\cM\in\Dc(\cD_{\AA^1})$ be a complex of coherent $\cD_{\AA^1}$-modules, and let $j$ be the canonical inclusion of $\gm$ into $\AA^1$. We define its Fourier localization as
$$\FLoc\cM=\FT^{-1}j_+j^+\FT\cM.$$
\end{defi}

\begin{nota}
The functor $\FLoc$ is exact in the category of coherent $\cD_{\AA^1}$-modules. In addition, for any holonomic $\cD_{\AA^1}$-module $\cM$, there exist two canonical morphisms $\cM\lra\FLoc\cM$ and $\NC\cM\lra\FLoc\cM$ whose kernels and cokernels are constant (we have the first one as well if $\cM$ is only coherent).
\end{nota}

\noindent\textit{Proof of Theorem \ref{interesante}.}
Since $j^+\FT\bar{K}_n$ is a single $\cD_{\gm}$-module at degree zero, by the previous remark all the other cohomologies of $\bar{K}_n$, and then of $\left(p_{n,+}\cO_{\cXn}\right)^G$, must be constant, so we can focus on the zeroth one.

After Remark \ref{semisimple} and Proposition \ref{FTKn}, we just need to show that $\iota_n^+\cH_n$ is isomorphic to $\NC_0j^+\cH^0\bar{K}_n$, which is isomorphic to $j^+\NC\cH^0\bar{K}_n$ because of $\cH^0\bar{K}_n$ not having singularities at the origin. We know that $j^+\FT\bar{K}_n\cong\cD_{\gm}/(P)$, where $P$ is
$$\gamma_n\prod_{i=0}^n\prod_{j=1}^{w_i}\left(D-\frac{d_n j}{w_i}\right)-\lambda^{d_n}.$$
Therefore, applying $\FT^{-1}$ to the canonical morphism $\cD_{\AA^1}/(P)\ra j_+j^+\FT\bar{K}_n$, we can deduce that there exists another canonical morphism from $\cM:=\FT^{-1}\cD_{\AA^1}/(P)=\cD_{\AA^1}/(Q)$ to $\FLoc\bar{K}_n=\FLoc\cH^0\bar{K}_n$. We can choose $Q$ as minus the image by $\FT^{-1}$ of $P$, that is,
$$Q=\dd^{d_n}-\gamma_n\prod_{i=0}^n\prod_{j=1}^{w_i} \left(\dd\lambda+\frac{d_n j}{w_i}\right).$$
Since the kernel and the cokernel of the morphism $\cD_{\AA^1}/(P)\ra j_+j^+\FT\bar{K}_n$ are necessarily supported at the origin, then both those of $\cM\ra\FLoc\bar{K}_n$ must be constant. In other words, we can form an exact sequence of the form
$$0\lra\cO_{\AA^1}^r\lra\cM\lra\FLoc\cH^0\bar{K}_n\lra\cO_{\AA^1}^s\lra0,$$
for some $r,s\geq0$. Since $\NC\FLoc=\NC$, applying now $j^+\NC$ to that sequence we obtain the isomorphism $j^+\NC\cM\cong j^+\NC\cH^0\bar{K}_n$.

Now note that $\cM$ has no singularities at the origin, so $j^+\NC\cM\cong\NC_0 j^+\cM$ as well. Then we just need to compute $j^+\cM$, but then we can multiply $Q$ on the left by $\lambda^{d_n}$, and since $\lambda^{d_n}\dd^{d_n}=\prod_{j=0}^{d_n-1}(D-j)$, we have in conclusion that $j^+\NC\cH^0\bar{K}_n$ is isomorphic to the nonconstant part of the quotient of $\cD_{\gm}$ by the left ideal generated by
$$\prod_{j=0}^{d_n-1}(D-j)-\gamma_n\lambda^{d_n}\prod_{i=0}^n\prod_{j=1}^{w_i} \left(\dd\lambda+\frac{d_n j}{w_i}\right),$$
which is isomorphic to the inverse image by $\iota_n$ of the hypergeometric $\cD_{\gm}$-module
$$\cH_{\gamma_n} \left(\frac{1}{w_0},\ldots,\frac{w_0}{w_0},\ldots, \frac{1}{w_n},\ldots,\frac{w_n}{w_n}; \frac{1}{d_n},\ldots,\frac{d_n}{d_n}\right).$$
This hypergeometric $\cD$-module has $\cH_n$ among its composition factors, together with every Kummer $\cD$-module of the form $\cK_\alpha$ with $\alpha\in C_n$ (cf. \cite[Cor. 3.7.5.2]{KaESDE}). However, all those Kummer factors are sent to $\cO_{\gm}$ by $\iota_n^+$, so the proof is finished.\hfill$\Box$

\vspace{.3cm}

Define the complex $K_n$ to be $p_{n,+}\cO_{\cZ_n-p_n^{-1}(0)}\in\Drh(\cD_{\gm})$. It fits with the statement of Theorem \ref{teorema2} by applying base change using diagram \ref{cartesiano YZ} (cf. \cite[Thm. 1.7.3]{HTT}). To prove that theorem we should compute then the Gauss-Manin complex of the family $\cZ_n$ outside the origin. Thanks to a combination mainly of Proposition \ref{isohyp} and Proposition \ref{hyperLS} we can characterize $\cG_n$ up to semisimplification if we prove that its Euler-Poincaré characteristic is -1, find its generic rank as $\cO_{\gm}$-module, calculate the exponents at the origin and infinity and know where in $\gm$ it has a singularity. The expression for the extension of the hypergeometric $\cD$-module with the Kummer ones will appear as an interesting consequence of one of the proofs below. We will summarize the strategy to determine $\cG_n$ in a more detailed way, but before, let us prove a small part of the main theorem. Its statement may seem vague and easy to prove, but it tells us just the information we will need at some moment in the future.

\begin{lema}\label{arreglo}
For any $n\geq2$, let $\bar{A}$ be the projective arrangement of hyperplanes in $\PP^{n-1}$ defined by
$$\bar{A}:x_1\cdot\ldots\cdot x_n(x_1+\ldots+x_n)=0\subset\PP^{n-1}.$$
Then for $-(n-2)\leq i\leq1$,
$$\cH^i\pi_{\PP^{n-1},+}\bR\Gamma_{\left[\bar{A}\right]}\cO_{\PP^{n-1}}\cong\left\{\begin{array}{ll}
\kk^{\binom{n}{i+n-2}} & \text{if }2\mid i+n,\\
\kk^{\binom{n}{i+n-2}+1} & \text{if }2\nmid i+n.\end{array}\right.$$
\end{lema}
\begin{proof}
From considering the diagram $\PP^{n-1}-\bar{A}\ra\PP^{n-1}\leftarrow\bar{A}$ and applying $\pi_{\PP^{n-1},+}$ we can obtain the distinguished triangle
$$\pi_{\PP^{n-1},+}\bR\Gamma_{\left[\bar{A}\right]}\cO_{\PP^{n-1}}\lra \pi_{\PP^{n-1},+}\cO_{\PP^{n-1}}\lra \pi_{\PP^{n-1},+}\cO_{\PP^{n-1}}(*\bar{A})\lra.$$
Note that $\PP^{n-1}-\bar{A}$ is also the complement of an affine arrangement $A$ of $n$ hyperplanes in general position in $\AA^{n-1}$. Therefore $\pi_{\PP^{n-1},+}\cO_{\PP^{n-1}}(*\bar{A})\cong\pi_{\AA^{n-1},+}\cO_{\AA^{n-1}}(*A)$, and by virtue of \cite[Prop. 5.2]{CaMV} and knowing the global de Rham cohomology of the projective space, the following fragments occur in the long exact sequence of cohomology of the triangle for $i+n$ odd and $-(n-2)\leq i\leq0$:
$$0\ra\kk^{\binom{n}{i+n-2}}\ra\cH^i\pi_{\PP^{n-1},+} \bR\Gamma_{\left[\bar{A}\right]}\cO_{\PP^{n-1}}\ra \kk\ra\kk^{\binom{n}{i+n-1}}\ra\cH^{i+1}\pi_{\PP^{n-1},+} \bR\Gamma_{\left[\bar{A}\right]}\cO_{\PP^{n-1}}\ra0.$$
If $i=1$ and $n+1$ is odd we can also extract the exact sequence
$$0\ra\kk^n\ra \cH^1\pi_{\PP^{n-1},+}\bR\Gamma_{\left[\bar{A}\right]}\cO_{\PP^{n-1}}\ra \kk\ra0.$$
The complex $\pi_{\PP^{n-1},+}\cO_{\PP^{n-1}}(*\bar{A})$ is isomorphic to the Orlik-Solomon algebra of the arrangement $\bar{A}$ (cf. \cite[Lem. 5]{Brieskorn}), which is generated by the inverse of the equations of each hyperplane in it. Therefore the morphism $\cH^i\pi_{\PP^{n-1},+}\cO_{\PP^{n-1}}\lra \cH^i\pi_{\PP^{n-1},+}\cO_{\PP^{n-1}}(*\bar{A})$ is zero and then the statement holds.
\end{proof}

\begin{prop}\label{cociente}
For every $n\geq1$, the following hold:
\begin{enumerate}[i)]
\item $\cH^0K_n$ has a quotient isomorphic to $\cO_{\gm}^n$.
\item For $-(n-2)\leq i\leq-1$ and $i+n$ odd, the generic ranks of $\cH^{i-1}K_n$ and $\cH^{i}K_n$ add up to $\binom{n+1}{i+n-1}$. Moreover, for any $i=-(n-1),\ldots,-1$,
    $$\binom{n}{i+n-1}-1\leq\operatorname{rk}\cH^iK_n\leq\binom{n}{i+n-1}+1$$
    and in any case $\operatorname{rk}\cH^{-1}K_n\leq\binom{n}{n-2}+1$.
\item The cohomologies $\cH^i\bar{K}_n$ are constant $\cD_{\AA^1}$-modules for $i\leq-1$.
\end{enumerate}
\end{prop}
\begin{proof}
Let us prove first point i. Let us consider, for the sake of simplicity, $\cZ_n\subset\AA^n\times\gm$, and let $\cZbar_n\subset\PP^n\times\gm$ and $\cZinf_n\subset\PP^{n-1}\times\gm$ be its projective closure in the first factor and its intersection with the hyperplane at infinity, respectively. Let us call now $\cM=\bR\Gamma_{\left[\cZbar_n\right]}\cO_{\PP^n\times\gm}$. Then we can form the distinguished triangle
$$\cM\lra j_+j^+\cM\lra i_+i^+\cM\lra$$
associated with the diagram $\AA^n\times\gm\stackrel{j}{\ra} \PP^n\times\gm\stackrel{i}{\leftarrow}\PP^{n-1}\times\gm$.

Let us apply $\pi_{2,+}$ to the above triangle to obtain a new one. By \cite[Prop. I.6.3.1]{Meb6Oper}, $j^+\cM\cong\bR\Gamma_{[\cZ_n]}\cO_{\AA^n\times\gm}$. Then, thanks to $\cZ_n$ being smooth, $\pi_{2,+}j^+\cM[1]$ is actually our $K_n$. The long exact sequence of cohomology of the new triangle contains the following piece:
$$\ldots\lra\cH^1\pi_{2,+}j^+\cM \lra\cH^1\pi_{2,+}i^+\cM\lra\cH^2\pi_{2,+}\cM\lra0.$$
We will have proved the statement of point i if we show the following:
$$\cH^1\pi_{2,+}i^+\cM\cong\left\{\begin{array}{lc}
\cO_{\gm}^{n+1} & \text{if }2\mid n\\
\cO_{\gm}^n & \text{if }2\nmid n\end{array}\right. \text{ and } \cH^2\pi_{2,+}\cM\cong\left\{\begin{array}{lc}
\cO_{\gm} & \text{if }2\mid n\\
0 & \text{if }2\nmid n\end{array}.\right.$$
By means of \textit{loc. cit.} again,
$$i^+\cM\cong \bR\Gamma_{[\cZinf_n]}\cO_{\PP^{n-1}\times\gm}\cong \pi_1^+\bR\Gamma_{\left[\bar{A}\right]}\cO_{\PP^{n-1}},$$
where $\bar{A}$ is the projective arrangement of hyperplanes of the previous lemma, in such a way that $\cZinf_n\cong\bar{A}\times\gm$. Then
$$\pi_{2,+}i^+\cM\cong\pi_{\PP^{n-1},+} \bR\Gamma_{\left[\bar{A}\right]}\cO_{\PP^{n-1}} \otimes_{\kk}\cO_{\gm},$$
so by Lemma \ref{arreglo} the first couple of isomorphisms holds.

Let us go for the second one, and let us go back to $\cXn$ to obtain them. We already know that each cohomology of $p_{n,+}\cO_{\cXn}$ is a semisimple $\cD_{\AA^1}$-module, and then the middle extension of its restriction to $U_n$, where $p_n$ is smooth. Then every fiber over a point of $U_n$ will have the same global de Rham cohomology as the fiber over the origin, which is a Fermat hypersurface. In particular, $\cH^1p_{n,+}\cO_{p_n^{-1}(\lambda_0)}\cong\kk$ if $n$ is even, vanishing if $n$ is odd. In that case, it is obvious that $\cH^1p_{n,+}\cO_{\cXn}$ will vanish too, but we cannot say anything for the moment about what happens when $n$ is even.

Restrict now our variety of parameters to $\gm$ and consider $\cXaf_{n,w}$ and $\cXinf_{n,w}$ to be the affine part of $\cXn$ and its intersection with the hyperplane at infinity within the first factor of $\PP^n\times\gm$. Consider the excision triangle associated with the diagram $\cXaf_{n,w}\stackrel{j}{\ra} \cXn\stackrel{i}{\leftarrow}\cXinf_{n,w}$ and the $\cD$-module $\cO_{\cXn}$ and apply $p_{n,+}$ to it, from which we can get the exact sequence (calling $p_n$ to the restrictions of such projection to $\cXaf_{n,w}$ and $\cXinf_{n,w}$ too)
$$\ldots\lra\cH^0p_{n,+}\cO_{\cXaf_{n,w}} \lra\cH^0p_{n,+}\cO_{\cXinf_{n,w}} \lra\cH^1p_{n,+}\cO_{\cXn}\lra0.$$
Note that $\cXinf_{n,w}$ is the cartesian product of another Fermat hypersurface with $\gm$, so $\cH^0p_{n,+}\cO_{\cXinf_{n,w}}$ is $\cO_{\gm}^r$ for some $r>0$. We have then that $\cH^1p_{n,+}\cO_{p_{n}^{-1}(U_n^*)}$ is a $\cD_{U_n^*}$-module of generic rank 1 and at the same time a quotient of $\cO_{U_n^*}^r$, so it is nothing but $\cO_{U_n^*}$. In conclusion, $\cH^1p_{n,+}\cO_{\cXn}\cong\cO_{\gm}$ if $n$ is even.

Now let us continue our journey passing from $\cXn$ to $\bar{\cY}_{n,w}$ and from there to $\cZbar_n$. Remember that there was an étale morphism between the last two defined by $\tilde{\alpha}_n((x_0:\ldots:x_n),\lambda)=((\lambda x_0:x_1:\ldots:x_n),\lambda^{-d_n})$. Since $\tilde{\alpha}_n$ can be extended to $\PP^n\times\gm$, we have that $\bR\Gamma_{\left[\bar{\cY}_{n,w}\right]} \cO_{\PP^n\times\gm}\cong \tilde{\alpha}_n^+\cM$ and by base change, $\pi_{2,+}\bR\Gamma_{\left[\bar{\cY}_{n,w}\right]} \cO_{\PP^n\times\gm}\cong\iota_n^+\pi_{n+1,+}\cM$.

We know that $\cH^2\pi_{2,+}\bR\Gamma_{\left[\bar{\cY}_{n,w}\right]} \cO_{\PP^n\times\gm}$ is the invariant part of $\cH^1p_{n,+}\cO_{\cXn}$ under the action of $G$, which is irreducible of rank one whenever it is nonzero. As a consequence, $\cH^2\iota_n^+\pi_{2,+}\cM$ must be zero if $n$ is odd, and $\cO_{\gm}$ if $n$ is even. If $n$ is odd we have proved what we wanted to, so let us take $n$ even. Since $\iota_n^+$ is an exact functor in the category of $\cD_{\gm}$-modules, $\cH^2\pi_{2,+}\cM$ must be a Kummer $\cD$-module, eventually trivial. But it is a quotient of $\cH^1\pi_{2,+}i^+\cM$, which is known to be a direct sum of copies of $\cO_{\gm}$, so it will also be $\cO_{\gm}$. This ends the proof of the second couple of isomorphisms.

Let us prove now points ii and iii. Recall that we are only interested in knowing the rank of $\cH^iK_n$; since $\iota_n$ is an étale map, we just need to show the analogous statement for $\iota_n^+K_n=\bar{K}_{n|\gm}=\pi_{2,+}\bR\Gamma_{[\cYn]}\cO_{\AA^n\times\gm}[1]$. Consider then the excision triangle associated with $\cN:=\bR\Gamma_{[\bar{\cY}_{n,w}]}\cO_{\PP^n\times\gm}$ and the diagram $\AA^n\times\gm\stackrel{j}{\ra} \PP^n\times\gm\stackrel{i}{\leftarrow}\PP^{n-1}\times\gm$. Analogously as before, if we apply $\pi_{2,+}$ to it, we will obtain the triangle
\begin{equation}\label{trianguloY}
\pi_{2,+}\cN\lra\bar{K}_n[-1]\lra\pi_{2,+}\bR\Gamma_{[\cY_{n,w}^{\infty}]}\cO_{\PP^{n-1}\times\gm}\lra.
\end{equation}
Recall that $\cY_{n,w}^{\infty}$ is the section at infinity of $\bar{\cY}_{n,w}$ and the product of the arrangement of hyperplanes $\bar{A}$ with $\gm$. Following the same argument used for the first point and using Lemma \ref{arreglo} again, we can claim that for $i\leq1$,
$$\cH^i\pi_{2,+}\bR\Gamma_{[\cY_{n,w}^{\infty}]}\cO_{\PP^{n-1}\times\gm}\cong \left\{\begin{array}{lc}
\cO_{\gm}^{\binom{n}{i+n-2}+1} & \text{if }2\nmid i+n\\
\cO_{\gm}^{\binom{n}{i+n-2}} & \text{if }2\mid i+n\end{array}\right.$$
Note that $\cH^i\pi_{2,+}\cN$ is the invariant part of $\cH^{i-1} p_{n,+}\cO_{\cXn}$ under the action of $G$. As said in the proof of the first point, the smooth fibers of $\cXn$ have the same cohomology as a Fermat hypersurface of $\PP^n$, so for $-(n-2)\leq i\leq0$, the cohomologies $\cH^i\pi_{2,+}\cN$ will vanish if $i+n$ is odd and will be $\cO_{\gm}$ if $i+n$ is even. Therefore, for $i\leq-1$ and $i+n$ odd we can extract from the triangle \ref{trianguloY} the exact sequences
$$0\ra\cH^{i-1}\bar{K}_n\ra\cO_{\gm}^{\binom{n}{i+n-2}+1}\ra\cH^{i+1}\pi_{2,+}\cN\ra \cH^{i}\bar{K}_n\ra\cO_{\gm}^{\binom{n}{i+n-1}}\ra0,$$
and if $i=-1$ and $i+n$ even,
$$0\lra\cH^{-1}\bar{K}_n\lra\cO_{\gm}^{\binom{n}{n-2}+1}.$$
The statement of point ii follows just by considering the cases in which $\cH^{i+1}\pi_{n+1,+}\cN$ maps either to zero or to a subobject of rank one of $\cH^i\bar{K}_n$, and point iii is just an easy consequence of the exact sequences above and Remark \ref{semisimple}.
\end{proof}

Let us construct $K_n$ in an alternative, more useful way. Let $\lambda_n$ be the morphism defined by
$$\begin{array}{rcl}\lambda_{n}:\AA^n&\longrightarrow&\AA^1\\
\underline{x}&\longmapsto& x_1^{w_1}\cdot\ldots\cdot x_n^{w_n}\cdot(1-x_1-\ldots-x_n)^{w_0}.\end{array}$$
Let $Z_n=\lambda_n^{-1}(\gm)=\left\{\underline{x}\in\gm^n\text{ : }x_1+\ldots+x_n\neq1\right\}$. Therefore, since $\cZ_n$ is a graph, we can take $K_n=\lambda_{n,+}\cO_{Z_n}$.

We will make use of an inductive process to be detailed in the next section, so let us obtain all the information about $K_1$ that we will make use of during such process.

\begin{lema}\label{inicio}
$K_1$ is a regular $\cD$-module over $\gm$ of generic rank $d_1$ and it has a unique singularity at $\gamma_1$.
\end{lema}
\begin{proof}
Let $C=\lambda_1^{-1}(\gamma_1)$. Then, $\lambda_1$ is an étale morphism from $Z_1-C$ to $\gm-\{\gamma_1\}$ of degree $d_1$, so $\lambda_{1,+}\cO_{Z_1-C}$ will actually be a unique $\cD_{\gm-\{\gamma_1\}}$-module; moreover, it will be a locally free $\cO_{\gm-\{\gamma_1\}}$-module of rank $d_1$, which will be the generic rank of $K_1$.

On the other hand, $Z_1=\gm-\{1\}$ and $\pi_{\gm,+}K_1=\pi_{\gm-\{1\},+}\cO_{\gm-\{1\}}$, so the Euler-Poincaré characteristic of $K_1$ will be equal to that of $\cO_{\gm-\{1\}}$, which is -1. Therefore, thanks to the additivity of the characteristic and Proposition \ref{hyperLS}, we will be able to find among the composition factors of $K_1$ some eventually trivial Kummer $\cD$-modules and an irreducible hypergeometric $\cD_{\gm}$-module (punctual or not), and so its only singular point within $\gm$ must be $\gamma_1$.
\end{proof}

One could wonder why the assumptions on $w_0,\ldots,w_n$ are like that. The first condition that we could try to erase is that all of the $w_i$ are positive. If not, for some $r\geq0$ and every $i=0,\ldots,r$ we would have that $w_i=0$. Under this assumption, the morphisms $'p_n$ and $''p_n$ would be smooth over the whole of $\AA^1$, and $K_n$ would be the direct image of $\cO_{\gm^n}$ by the morphism $\lambda_n(\xx)=x_{r+1}^{w_{r+1}}\cdot\ldots\cdot x_n^{w_n}$. This context is already discussed in \cite[\S4]{CaKoszul}, and we would still have that $\bar{K}_n$ is constant, and even $K_n$, if $\gcd(w_{r+1},\ldots,w_n)=1$.

So that case is not relevant, but we could also consider the case in which we had that $\gcd(w_0,\ldots,w_n)>1$. Then there would be an integer $e$ dividing all of the $w_i$, so that $G$ and $\cYn$ would be the disjoint union of their irreducible components, all of them differing only by a $e$-th root of unity. Going downstairs to the context of $\cZ_{n,w}$ and back to $\cYn$, we would have that
$$'p_{n,+}\cO_{\cYn}\cong\bigoplus_{\zeta\in\mu_e}h_{\zeta,+} {}'p_{n,+}\cO_{\cY_{n,w/e}},$$
where $h_{\zeta}$ is the homothety of $\AA^1$ defined by $\lambda\mapsto\zeta\lambda$, so in the end, we could know $\bar{K}_n$ by computing $'p_{n,+}\cO_{\cY_{n,w/e}}$, reducing the calculation to the original setting.

\section{Inductive process}\label{induccion}

In this section we will move forward towards the proof of Theorem \ref{teorema2}, finding some of the desired properties of $K_n$. All of the proofs are inductive, and in fact, together with those of the next section, can be conceived as a long proof divided into several pieces, each of them being approximately a sentence of the statement of the theorem. In spite of the discussion at the last paragraph of the previous section about the case in which $\gcd(w_0,\ldots,w_n)>1$, we should be able to allow this fact when working with $K_n$, for as the reader might notice, we can have a $(n+1)$-tuple $(w_0,\ldots,w_n)$ such that every sub $n$-tuple resulting from it by taking a single element off shares a common divisor ($(6,10,15)$ for instance). This does not alter the validity of the following propositions, and although we are not so interested in that case, we will cover it as well, denoting by $e_r$ the greatest common divisor of $w_0,\ldots,w_r$, for any value of $r$.

Despite the interdependence of the propositions of this section because of the induction, the reader will be able to check that we make no circular reasoning. We will explain this in more detail.

Let us factor $\lambda_n$ through $\gm^2$ as
$$Z_n\stackrel{(\pi_n,\lambda_n)}{\longrightarrow}\gm^2 \stackrel{\pi_2}{\longrightarrow}\gm,$$
so that $K_n=\pi_{2,+}(\pi_n,\lambda_n)_+\cO_{Z_n}$. Call $L_n$ the complex $(\pi_n,\lambda_n)_+\cO_{Z_n}$.

Consider now the isomorphisms $\phi_n$ from $(\gm-\{1\})\times\gm$ to itself given by $(z,\lambda)\mapsto(z,\lambda/(z^{w_n}(1-z)^{d_{n-1}}))$, and $\psi_n$ from $Z_n-\{x_n=1\}$ to $Z_{n-1}\times(\gm-\{1\})$ given by $(x_1,\ldots,x_n)\mapsto(x_1/(1-x_n),\cdots,x_{n-1}/(1-x_n),x_n)$. Those morphisms form the cartesian diagram
$$\xymatrixcolsep{3pc}\xymatrix{\ar @{} [dr] |{\Box} Z_n-\{x_n=1\} \ar[d]_{(\pi_n,\lambda_n)} \ar[r]^{\psi_n} & Z_{n-1}\times(\gm-\{1\}) \ar[d]^{\pi_2\times\lambda_{n-1}\pi_1}\\
(\gm-\{1\})\times\gm \ar[r]^{\phi_n} & (\gm-\{1\})\times\gm},$$
so by base change and the Künneth formula, $L_{n|(\gm-\{1\})\times\gm}\cong(\pi_2\phi_n)^+K_{n-1}$. This expression does not give us all the information about $K_n$ directly from $K_{n-1}$; to solve this issue we act as follows.

For each $n$, the process of finding $K_n$ depends on two inductive steps. Let us write, for each $n\geq2$, $T_n=\left\{\underline{x}\in\gm^{n-1}\,\text{ : }\, x_1+\ldots+x_{n-1}\neq0\right\}$. Each $T_n$ can be seen as a smooth closed subvariety of $Z_n$ by the identification $T_n\cong T_n\times\{1\}$, and we will do that in what follows. From the diagram $Z_n-T_n\stackrel{j}{\ra}Z_n\stackrel{i}{\leftarrow}T_n$ we can get the triangle
$$\cO_{Z_n}\lra j_+j^+\cO_{Z_n}\lra i_+i^+\cO_{Z_n}\lra.$$
Applying $(\pi_n,\lambda_n)_+$ to it we get a new one:
$$(\pi_n,\lambda_n)_+\cO_{Z_n}\longrightarrow (\pi_n,\lambda_n)_+j_+j^+\cO_{Z_n} \longrightarrow(\pi_n,\lambda_n)_+i_+i^+\cO_{Z_n}\lra.$$
In other words, defining $M_n:=\lambda_{n|T_n,+}\cO_{T_n}$, we have
$$L_n\longrightarrow j_+(\pi_2\phi_n)^+K_{n-1} \longrightarrow i_+M_n\lra,$$
where $j$ and $i$ now stand for the inclusions $(\gm-\{1\})\times\gm\hra\gm^2\leftarrow\{1\}\times\gm$ and $M_n$ tells us what we lose when doing induction over $\left(\gm-\{1\}\right)\times\gm$ instead of $\gm^2$. We will calculate its expression later on. What we are interested in is noting that $K_n$ depends on $K_{n-1}$ and $M_n$, for applying $\pi_{2,+}$ to the last triangle we obtain the one that is going to be useful for us:
\begin{equation}\label{trianguloclave}
K_n\longrightarrow \pi_{2,+}(\pi_2\phi_n)^+K_{n-1} \longrightarrow M_n\lra.
\end{equation}
In the next proposition we will see that our method to make $M_n$ explicit depends only on $K_{n-2}$, and thus the induction can be made correctly. Before going on, we will provide some useful lemmas:

\begin{lema}\label{descdirec}
Let $X=Y\times Z$ be the product of two smooth affine varieties such that $Z$ is of dimension one. Let $K$ be a complex of coherent $\cD_X$-modules. Then for any integer $i$ we have the exact sequence
$$0\lra\cH^0\pi_{1,+}\cH^iK\lra \cH^i\pi_{1,+}K\lra\cH^{-1}\pi_{1,+}\cH^{i+1}K\lra0.$$
\end{lema}
\begin{proof}
Let us fix $i$, consider the truncation triangle
$$\tau_{\leq i}K\lra K\lra\tau_{\geq i+1}K\lra$$
and apply $\pi_{1,+}$ to it. Since $Y$ and $Z$ are affine, $\cH^k\pi_{1,+}\tau_{\leq i}K=0$ and $\cH^l\pi_{1,+}\tau_{\geq i+1}K=0$ for any $k>i$ and $l<i$, respectively, because $\pi_1$ is of relative dimension one and $\pi_{1,+}$ is just taking a relative de Rham complex. Moreover, we can also deduce that $\cH^i\pi_{1,+}\tau_{\leq i}K=\cH^0\pi_{1,+}\cH^iK$ and $\cH^i\pi_{1,+}\tau_{\geq i+1}K=\cH^{-1}\pi_{1,+}\cH^{i+1}K$. Therefore, the long exact sequence of cohomology of the triangle above contains in the $i$-th degree the piece we were looking for.
\end{proof}

\begin{lema}\label{sicigiahom}
Let $w_0,\ldots,w_n$ be an $(n+1)$-tuple of positive integers, whose sum is $d_n$, and let $f=x_1^{w_1}\cdot\ldots\cdot x_n^{w_n}(x_1+\ldots+x_n)^{w_0}$. Then, the syzygies of the Jacobian ideal $J_f=(f,f'_1,\ldots,f'_n)\subseteq\kk[\xx]$ are generated as a $\kk[\xx]$-submodule of $\kk[\xx]^{n+1}$ by the Euler relation $(-d_n,x_1,\ldots,x_n)$ and the Koszul-like syzygies
$$\frac{x_ix_j(x_1+\ldots+x_n)}{f}(f'_je_i-f'_ie_j),\text{ for all }1\leq i<j\leq n.$$
\end{lema}
\begin{proof}
For the sake of simplicity, let us write $\xx^w$ and $\sigma$ for $x_1^{w_1}\cdot\ldots\cdot x_n^{w_n}$ and $x_1+\ldots+x_n$, respectively, so that $f=\xx^w\sigma^{w_0}$, and $l_i=w_i\sigma+w_0x_i$ for each $i=1,\ldots,n$, in such a way that $f'_i=\xx^{w-e_i}\sigma^{w_0-1}l_i$.

$f$ is a homogeneous polynomial of degree $d_n$, so the Euler syzygy appears naturally among the generators because of having its first component of degree zero. That is why we can restrict ourselves to find the syzygies of $(f'_1,\ldots,f'_n)$. Let $(a_1\ldots,a_n)\in\kk[\xx]^n$ such that $\sum_if'_ia_i=0$, or in other words
$$\sigma^{w_0-1}\sum_{i=1}^na_i\xx^{w-e_i}l_i=0.$$
This means that $(a_1l_1,\ldots,a_nl_n)$ is a syzygy of the ideal $(\xx^{w-e_1},\ldots,\xx^{w-e_n})$, so for each $i$, $x_i$ must divide $a_il_i$, and thus there will exist $n$ new polynomials $b_1,\ldots,b_n$ such that $a_i=x_ib_i$ for every $i$, because $(x_i,l_i)=1$. Therefore, $(b_1,\ldots,b_n)$ is a syzygy of $(l_1,\ldots,l_n)$, which is a regular sequence in $\kk[\xx]$, just because it is an isomorphic image of $(x_1,\ldots,x_n)$. Indeed, the matrix of the change of basis has a determinant equal to $w_0^{n-1}d_n$ by Sylvester's determinant theorem. As a consequence, there must exist $g_{(i,j)}\in\kk[\xx]$ for every couple $1\leq i<j\leq n$ such that
$$(b_1,\ldots,b_n)=\sum_{1\leq i<j\leq n}g_{(i,j)}(l_je_i-l_ie_j),$$
and then,
$$(a_1,\ldots,a_n)=\sum_{1\leq i<j\leq n}g_{(i,j)}\frac{x_ix_j\sigma}{f}(f'_je_i-f'_ie_j).$$
\end{proof}

\begin{prop}\label{homogeneo}
For each $n$, $M_n$ is concentrated in degrees $-(n-2)$ to zero, $$\cH^iM_n\cong\bigoplus_{a=1}^{e_{n-1}}\cK_{a/e_{n-1}}^{\binom{n-1}{i+n-2}}$$
whenever $-(n-2)\leq i\leq-1$, and in degree zero,
$$\cH^{0}M_n\cong\bigoplus_{a=1}^{e_{n-1}}\cK_{a/e_{n-1}}^{n-2}\oplus \bigoplus_{a=1}^{d_{n-1}}\cK_{a/d_{n-1}}.$$
\end{prop}
\begin{proof}
Before starting, let us assume that $e_{n-1}=1$; the statement of the proposition follows easily from the proof of this particular case by taking $[e_{n-1}]_+$. When $n=2$, $M_n\cong[d_1]_+\cO_{\gm}\cong\bigoplus_{i=1}^{d_1}\cK_{i/d_1}$ and we are done, so assume from now on that $n>2$.

Let $\phi:Z_{n-2}\times\gm\ra T_n$ be the isomorphism given by $\phi(\xx,\lambda)=(\lambda\xx,-\lambda)$. With this notation, $M_n\cong(\lambda_n\phi)_+\cO_{Z_{n-2}\times\gm}$. We will now decompose $\lambda_n\phi$ in a suitable way so that we can easily calculate $M_n$.

We can write $\lambda_n\phi=\pi_2\psi\left(\lambda_{n-2}\times[d_{n-1}]\right)$, with $\psi$ being the automorphism of $\gm^2$ defined by $\psi(x,y)=(x,(-1)^{w_{n-1}}xy)$. Therefore,
$$M_n\cong(\pi_2\psi)_+\left(\lambda_{n-2}\times[d_{n-1}]\right)_+\left(\cO_{Z_{n-2}}\boxtimes\cO_{\gm}\right) \cong\bigoplus_{a=1}^{d_{n-1}}(\pi_2\psi)_+\left(K_{n-2}\boxtimes\cK_{a/d_{n-1}}\right),$$
by the Künneth formula.

Now using the fact that $\psi$ is an automorphism, for each $a=1,\ldots,d_{n-1}$ we have that
$$(\pi_2\psi)_+\left(K_{n-2}\boxtimes\cK_{a/d_{n-1}}\right)\cong \pi_{2,+}\left(\left(\pi_1\psi^{-1}\right)^+K_{n-2}\otimes_{\cO_{\gm^2}} \left(\pi_2\psi^{-1}\right)^+\cK_{a/d_{n-1}}\right).$$
Now $\pi_1\psi^{-1}$ is nothing but the first projection of $\gm^2$, and $\left(\pi_2\psi^{-1}\right)^+\cK_{a/d_{n-1}}$ is easy to calculate, namely it is $\cK_{-a/d_{n-1}}\boxtimes\cK_{a/d_{n-1}}$.

Putting everything together and again by the Künneth formula,
$$M_n\cong\bigoplus_{a=1}^{d_{n-1}}\pi_{\gm,+}\left(K_{n-2}\otimes\cK_{-a/d_{n-1}}\right) \otimes_\kk\cK_{a/d_{n-1}}.$$

Let us now compute the summands above. When $a=d_{n-1}$, we are dealing with the constant part of $M_n$. However, we can calculate it in a easier way. The global de Rham cohomology of $M_n$ is that of $T_n$, which analogously to \cite[Prop. 5.2]{CaMV} can be computed to be
$$\bigoplus_{i=-(n-1)}^{-1}\kk^{\binom{n}{i+n-1}}[-i]\oplus\kk^{n-1}[0].$$
Then if we write the constant part of the cohomology of $M_n$ as $\cO_{\gm}^{c_i}$, for $-(n-2)\leq i\leq0$, we deduce that $c_{-(n-2)}=1$, $c_0=n-1$ and $c_i+c_{i+1}=\binom{n}{i+n-1}$ for every $i=-(n-2),\ldots,-1$. Therefore, $c_i=\binom{n-1}{i+n-2}$.

Assume now that $a\neq d_{n-1}$ and write $\alpha=-a/d_{n-1}$ for the sake of simplicity. We can make use of Lemma \ref{descdirec} to write
\begin{equation}\label{kummerMn}
\cH^i\pi_{\gm,+}\left(K_{n-2}\otimes\cK_\alpha\right)\cong \cH^0\pi_{\gm,+}\left(\cH^iK_{n-2}\otimes\cK_\alpha\right) \oplus \cH^{-1}\pi_{\gm,+}\left(\cH^{i+1}K_{n-2}\otimes\cK_\alpha\right).
\end{equation}
For $i=-(n-3),\ldots,-1$, the $i$-th cohomology of $K_{n-2}$ is
$$\cH^iK_{n-2}=\bigoplus_{a=1}^{e_{n-2}}\cK_{a/e_{n-2}}^{\binom{n-2}{i+n-3}},$$
so $\cH^iK_{n-2}\otimes\cK_\alpha$ will provide a nonzero global de Rham cohomology only if we have a Kummer $\cD$-module $\cK_{-\alpha}$ in $K_{n-2}$ such that both $e_{n-2}\alpha$ and $d_{n-1}\alpha$ are integers. That is equivalent to $w_i\alpha$ being an integer for every $i=0,\ldots,n-1$, that is to say, $e_{n-1}\alpha\in\ZZ$. But $e_{n-1}=1$, so no Kummer part at a degree below zero contributes to $M_n$, and so, for the values of $a$ under consideration and $i\leq-2$, $\cH^i\pi_{\gm,+}\left(K_{n-2}\otimes\cK_\alpha\right)=0$. In other words, $\cH^iM_n=\cO_{\gm}^{\binom{n-1}{i+n-2}}$ for $i=-(n-2),\ldots,-2$.

Using decomposition \ref{kummerMn} and the vanishing of the $\pi_{\gm,+}\left(\cH^j K_{n-2}\otimes\cK_\alpha\right)$ for $j\leq-1$, we can write
$$\cH^iM_n\cong\cO_{\gm}^{\binom{n-1}{i+n-2}}\oplus \bigoplus_{a=1}^{d_{n-1}-1}\cH^i\pi_{\gm,+}\left(\cH^0K_{n-2}\otimes\cK_{-a/d_{n-1}}\right) \otimes_\kk\cK_{a/d_{n-1}},$$
for $i=-1,0$. Let us remain in the case $a\neq d_{n-1}$. At degree zero, $\cH^0K_{n-2}$ lies in the middle of a short exact sequence. The right hand side is $\bigoplus_{a=1}^{e_{n-2}}\cK_{a/e_{n-2}}^{n-2}$, and we have already seen that it vanishes when we apply the functor $\pi_{\gm,+}\left(\bullet\otimes\cK_\alpha\right)$ to it. Therefore, keeping up with $M_n$, we have that for $i=-1,0$,
\begin{equation}\label{cohomMnGn}
\cH^iM_n\cong\cO_{\gm}^{\binom{n-1}{i+n-2}}\oplus \bigoplus_{a=1}^{d_{n-1}-1}\cH^i\pi_{\gm,+}\left(\cG_{n-2}\otimes\cK_{-a/d_{n-1}}\right) \otimes_\kk\cK_{a/d_{n-1}}.
\end{equation}
The composition factors of $\cG_{n-2}$ are $\cH_{n-2}$ and some Kummer $\cD$-modules. Since $\cH_{n-2}$ is an irreducible $\cD_{\gm}$-module of Euler-Poincaré characteristic $-1$, we can affirm that $\Ext_{\cD_{\gm}}^1(\cH_{n-2},\cK_\beta)\cong \Ext_{\cD_{\gm}}^1(\cK_\beta,\cH_{n-2})$ is an one-dimensional $\kk$-vector space, so up to our actual knowledge, $\cG_{n-2}$ is placed in the middle of a short exact sequence of one of the forms
\begin{equation}\label{sucesionescoro}
\begin{array}{c}
\displaystyle0\lra\cH_{n-2}\oplus\bigoplus_{\beta\in C'_{n-2}}\cK_\beta\lra\cG_{n-2}\lra \bigoplus_{\beta\in C''_{n-2}}\cK_\beta\lra0,\\[4pt]
\displaystyle0\lra\bigoplus_{\beta\in C'_{n-2}}\cK_\beta\lra\cG_{n-2}\lra \cH_{n-2}\oplus\bigoplus_{\beta\in C''_{n-2}}\cK_\beta\lra0,
\end{array}
\end{equation}
where $C'_{n-2}$ and $C''_{n-2}$ are two possibly empty, complementary subsets of $C_{n-2}$. The hypergeometric irreducible composition factor remains irreducible and hypergeometric when tensored with a Kummer $\cD$-module, so it will only contribute with one copy of $\kk$ at degree zero when we apply $\pi_{\gm,+}\left(\bullet\otimes\cK_\alpha\right)$ to it. It is easy to check that independently on $a\neq d_{n-1}$ and the position of the hypergeometric and Kummer composition factors on the exact sequence of $\cG_{n-2}$,
$$\cH^0\pi_{\gm,+}\left(\cG_{n-2}\otimes\cK_{-a/d_{n-1}}\right)\cong \cH^{-1}\pi_{\gm,+}\left(\cG_{n-2}\otimes\cK_{-a/d_{n-1}}\right)\oplus\kk,$$
so that we can find at $\cH^0M_n$ a copy of each $\cK_{a/d_{n-1}}$ more than at $\cH^{-1}M_n$.

In order to determine which Kummer $\cD$-modules appear at $\cH^0M_n$ (and then at $\cH^{-1}M_n$ too), we will apply \cite[Thm. 1.1]{CaKoszul} and try to study the surjectivity of the map $\Phi=(f-t,\dd_1+f'_1\fia,\ldots, \dd_{n-1}+f'_{n-1}\fia):R^n\ra R$ for non-integer values of $\alpha$, where $R=\kk((t))[x_1,\ldots,x_{n-1}]$, $f=x_1^{w_1}\cdot\ldots\cdot x_{n-1}^{w_{n-1}}(x_1+\ldots+x_{n-1})^{w_0}$ and $\fia=\dd_t+\alpha t^{-1}$. To see that we can assume, without loss of generality, that we take $c\in R$, homogeneous in the $x_i$ of degree $m\geq0$, and $a,b^1,\ldots,b^{n-1}\in R$, homogeneous too of degrees $m,m+1,\ldots,m+1$, respectively, such that $\Phi(a,b^1,\ldots,b^{n-1})=c$. Then what we are assuming reads
$$\left\{\begin{array}{rcl}
\displaystyle fa+\sum_{i=1}^{n-1}f'_i\fia b^i&=&0\\
\displaystyle-ta+\sum_{i=1}^{n-1}\dd_ib^i&=&c\end{array}.\right.$$
Thanks to Lemma \ref{sicigiahom}, from the first equation we know that there exist homogeneous polynomials $F, g_{(i,j)}\in R$ for every $1\leq i<j\leq n$, of respective degrees $m$ and $m-1$ in the $x_i$ (so that the $g_{(i,j)}$ must be zero if $m=0$), such that
\begin{equation}\label{sicigiasMn}
\begin{array}{l}
a=-d_{n-1}F\\
\displaystyle\fia b^i=x_iF+\sum_{j\neq i}\varepsilon(j-i)\frac{x_i x_j\sigma}{f}f'_jg_{(i,j)}\,,i=1,\ldots,n,\end{array}
\end{equation}
$\varepsilon$ being the sign operator. Note that $\fia$ is always bijective for the values of $\alpha$ under consideration. Substituting those values in the second equation of \ref{sicigiasMn} and applying the Euler relation for $F$, we get that
\begin{equation}\label{exponentesMn}
d_{n-1}\left(t+\frac{m+n}{d_{n-1}}\fiai\right)F+\sum_{1\leq i<j\leq n}L_{(i,j)}g_{(i,j)}=c,
\end{equation}
where
$$L_{(i,j)}=\left((w_j-w_i)\sigma+ (w_0-w_i)x_j-(w_0-w_j)x_i+x_il_j\dd_i-x_jl_i\dd_j\right)\fiai$$
for each pair $(i,j)$; they depend on $\alpha$ only because they apply $\fiai$ (which is always an isomorphism) to the $g_{(i,j)}$.

Call $A$ the operator acting on $F$. Now if $d_{n-1}\alpha$ is not an integer, $A$ is invertible, so the system has a solution. And when $d_{n-1}\alpha$ is an integer the lack of surjectivity is independent of the choice of $\alpha$, as we wanted to know; should equation \ref{exponentesMn} above have a solution, it would happen for any $\alpha$, changing appropriately $m$. In conclusion, all the classes modulo $\ZZ$ of the elements of $\{1,\ldots,d_{n-1}-1\}$ appear as exponents of $\cH^0M_n$ with the same multiplicity.

Then we can affirm that there exists some nonnegative integer $r$ such that, for $i=-1,0$,
$$\cH^iM_n\cong\cO_{\gm}^{\binom{n-1}{i+n-2}}\oplus\bigoplus_{a=1}^{d_{n-1}-1}\cK_{a/d_{n-1}}^{r+1+i}.$$
If we show that $r=0$ we will have finished the proof. Recall formula \ref{cohomMnGn} and how the Kummer summands are added to the last two cohomologies of $M_n$. At degree -1 they appear when tensoring some $\cK_\beta$, for $\beta\in C_{n-2}$, with some $\cK_\alpha$, for $d_{n-1}\alpha\in\ZZ$. We have just seen by the calculation of the exponents that if $r>0$ every Kummer $\cK_{a/d_{n-1}}$ should appear at $\cH^{-1}M_n$. However, note that $1/d_{n-1}$ can never be in $C_{n-2}$, for $w_i<d_{n-1}$ for any $i=0,\ldots,n-2$. Then $r=0$ and we are done.
\end{proof}

In the proof above we have stated in \ref{sucesionescoro} the two possibilities for the extension of $\cH_{n-2}$ and the Kummer $\cD$-modules to obtain $\cG_{n-2}$. \textit{A priori} it was not possible to determine its exact expression, but we can revisit the proof with all the information we know now about $M_n$ and be much more precise:

\begin{coro}\label{corosucGn}
For each $n>2$, we have an exact sequence of the form
$$0\lra\cH_{n-2}\lra\cG_{n-2}\lra\bigoplus_{\alpha\in C_{n-2}}\cK_{\alpha}\lra0.$$
\end{coro}
\begin{proof}
Recall the exact sequences of \ref{sucesionescoro} for the index $n-2$:
$$\begin{array}{c}
\displaystyle0\lra\cH_{n-2}\oplus\bigoplus_{\beta\in C'_{n-2}}\cK_\beta\lra\cG_{n-2}\lra \bigoplus_{\beta\in C''_{n-2}}\cK_\beta\lra0,\\[4pt]
\displaystyle0\lra\bigoplus_{\beta\in C'_{n-2}}\cK_\beta\lra\cG_{n-2}\lra \cH_{n-2}\oplus\bigoplus_{\beta\in C''_{n-2}}\cK_\beta\lra0.
\end{array}$$
Take $w_{n-1}=d_{n-2}$, so that every $\beta\in C_{n-2}$ appears as well as some $a/d_{n-1}$ for certain values of $a\in\{1,\ldots,d_{n-1}-1\}$. Then, if $C'_{n-2}$ were not empty, all of the
$\cH^{-1}\pi_{\gm,+}(\cK_\beta\otimes\cK_{-\beta})\otimes\cK_\beta\cong\cK_\beta$, for $\beta\in C'_{n-2}$, would be subobjects of $\cH^{-1}\pi_{\gm,+}(\cG_{n-2}\otimes\cK_{-\beta})\otimes\cK_\beta$, just by looking at the long exact sequence of cohomology of the triangle appearing from applying $\pi_{\gm,+}(\bullet\otimes\cK_{-\beta})\otimes\cK_\beta$ to one of the exact sequences of \ref{sucesionescoro}. However, $\cH^{-1}\pi_{\gm,+}(\cG_{n-2}\otimes\cK_{-\beta})\otimes\cK_\beta=0$, so necessarily $C'_{n-2}$ is empty and $C''_{n-2}$ is the whole $C_{n-2}$.

Now we have two possibilities: either the extension of the Kummer $\cK_\beta$ and the hypergeometric $\cH_{n-2}$ is trivial or it is not, corresponding respectively to the second and the first exact sequence of \ref{sucesionescoro}. Anyhow, we can write the exact sequence with $\cH_{n-2}$ on the left-hand side, independently of whether it splits or not.
\end{proof}

Even though the corollary provides us a piece of Theorem \ref{teorema2}, we will not be able to use it in the following, since it lets us find the exact sequence of $\cG_{n-2}$, two steps behind the induction procedure. We will only use it at the very end of the actual proof of the theorem in section \ref{fourier}.

We can also state an interesting small byproduct of Proposition \ref{homogeneo}, relating hyperplane arrangements:

\begin{coro}
For any $n\geq2$, let $\mathcal{A}=\{H_1,\ldots,H_{n+1}\}$ be a generic arrangement of $n+1$ hyperplanes in $\AA^n$ (i. e., a central arrangement such that the intersection of every subset of $n$ hyperplanes is the origin) with multiplicities $w_1,\ldots,w_{n+1}$, sharing no common factor and denoting by $d$ their sum. Then, the global de Rham cohomology of its Milnor fiber $F$ is
$$\cH^i\pi_{F,+}\cO_F\cong\kk^{\binom{n}{i+n-1}},\text{ for } -(n-1)\leq i\leq-1,\text{ and }\, \cH^0\pi_{F,+}\cO_F\cong\kk^{n+d-1}$$
\end{coro}
\begin{proof}
Under a suitable affine change of variables, we can assume that the form defining our arrangement is $f=x_1^{w_1}\cdot\ldots\cdot x_n^{w_n}(x_1+\ldots+x_n)^{w_{n+1}}$. Then, since it is homogeneous, the Milnor fiber of the arrangement is defined by the equation $f=1$. Let us denote by $\cT$ the subvariety of $\AA^n\times\gm$ given by $f-\lambda=0$, and consider the following cartesian diagram:
$$\xymatrixcolsep{3pc}\xymatrix{\ar @{} [dr] |{\Box} F\times\gm \ar[d]_{\pi_{2}} \ar[r]^{\alpha} & \cT \ar[d]^{\pi_{2}}\\
 \gm \ar[r]_{[d]} & \gm},$$
where $\alpha$ is the morphism given by $(\xx,\lambda)\mapsto(x_1/\lambda,\ldots,x_{n-1}/\lambda,\lambda^d)$ and the morphisms denoted by $\pi_{2}$ are actually the restriction to $F\times\gm$ or $\cT$ of the second projection of $\AA^n\times\gm$. All the varieties involved are smooth, so by base change,
$$[d]^+\pi_{2,+}\cO_\cT\cong \pi_{2,+}\cO_{F\times\gm}\cong\pi_{F,+}\cO_F\otimes_{\kk}\cO_{\gm}.$$
Then the formula for the global de Rham cohomology of $F$ follows just from the statement of the proposition, noting that $M_{n+1}\cong\pi_{2,+}\cO_{\cT}$.
\end{proof}

\begin{lema}\label{Kummerfallo}
Let $n\geq2$, and let $\alpha$ be an element of $\kk$ such that $d_{n-1}\alpha$ is an integer. Then,
$$\pi_{2,+}(\pi_2\phi_n)^+\cK_\alpha\cong\left\{\begin{array}{ll}
\cO_{\gm}[1]\oplus\cO_{\gm}^2 & \text{ if }\alpha\in\ZZ\\
\cK_\alpha & \text{ otherwise.}\end{array}\right.$$
\end{lema}
\begin{proof}
For any $\alpha$, a simple calculation shows that
$$(\pi_2\phi_n)^+\cK_\alpha\cong\cD_{\gm\times(\gm-\{1\})}/ (D_\lambda-\alpha,\dd_z-(d_nz-w_n)(z(1-z))^{-1}D_{\lambda}).$$
Note that the ideal can be rewritten as $(D_\lambda-\alpha,\dd_z-\alpha(d_nz-w_n)(z(1-z))^{-1})$, so that our $\cD$-module is in fact $\cK_\alpha\boxtimes k^+\cH_1(-\alpha w_n;-\alpha d_n)$, where $k:\gm-\{1\}\hra\gm$ is the canonical inclusion.

But then, $\pi_{2,+}(\pi_2\phi_n)^+\cK_\alpha\cong\cK_{\alpha}\otimes\pi_{\gm-\{1\},+}k^+\cH_1(-\alpha w_n;-\alpha d_n)$. Note that $(d_n-w_n)\alpha=d_{n-1}\alpha\in\ZZ$. Then the question is now how to calculate the global de Rham cohomology of some $k^+\cH_1(\alpha;\beta)$, for $\alpha,\beta\in\kk$ such that $\alpha\equiv\beta\!\mod\ZZ$. However, due to the condition on $\alpha$ and $\beta$, $k^+\cH_1(\alpha;\beta)\cong k^+\cK_\alpha$.

If $\alpha\in\ZZ$, then $k^+\cK_\alpha$ is just $\cO_{\gm-\{1\}}$, so $\pi_{\gm-\{1\},+}k^+\cK_\alpha\cong\kk[1]\oplus\kk^2$. In order to treat the case $\alpha\notin\ZZ$, we will make use of \cite[Prop. 2.9.8]{KaESDE} and extend our $\cD$-modules to $\gm$, for $\pi_{\gm-\{1\},+}k^+\cK_\alpha\cong\pi_{\gm,+}k_+k^+\cK_\alpha$. In that case, $\cK_\alpha$ is an irreducible $\cD_{\gm}$-module and has an integer exponent at 1 because it is regular there, so we can apply the aforementioned result to obtain the exact sequence
$$0\lra\cK_\alpha\lra k_+k^+\cK_\alpha\lra\delta_1\lra0.$$
Now since $\pi_{\gm,+}\cK_\alpha=0$, we get from the sequence that $\pi_{\gm,+}k_+k^+\cK_\alpha\cong\kk$.
\end{proof}

\begin{prop}\label{constkn}
For each $n\geq1$, $K_n$ is concentrated in degrees $-(n-1),\ldots,0$, $\cH^iK_n\cong\bigoplus_{a=1}^{e_n}\cK_{a/e_n}^{\binom{n}{i+n-1}}$ for all $-(n-1)\leq i\leq-1$ and in degree zero, we have an exact sequence of the form
$$0\lra\cG_n\lra \cH^{0}K_n\lra\bigoplus_{a=1}^{e_n}\cK_{a/e_n}^n\lra0,$$
where $\cG_n$ is a $\cD_{\gm}$-module of Euler-Poincaré characteristic $-1$, without any constant composition factor.
\end{prop}
\begin{proof}
Acting in the same way as in the proof of Proposition \ref{homogeneo}, let us assume that $e_n=1$; the general case can be easily deduced from this particular one. The statement is an improvement of that of Proposition \ref{cociente}. To prove the part regarding the zeroth cohomology we only need to show that $\cH^{0}K_n$ has only $n$ copies of $\cO_{\gm}$ as composition factors, appearing then all of them as a quotient, defining the $\cG_n$ as the $\cD$-module with which we take such a quotient.

To deal in a better way with the cohomologies of $K_n$, let us separate the copies of $\cO_{\gm}$ from the semisimplification of the rest of composition factors, such that $\left(\cH^iK_n\right)^{\text{ss}}=\cO^{c_i}\oplus G_i$, that is to say, we will denote by $G_i$ the direct sum of the nonconstant composition factors of $\cH^iK_n$.

It is straightforward to deduce that $\cH^iK_n=0$ for every $i\notin\{-n,\ldots,0\}$ decomposing $\lambda_n$ as the graph immersion followed by the second projection. Following the spirit of the section, let us proceed inductively in $n$ to find the $c_i$ and the $G_i$. If $n=1$ we have already proved everything at Lemma \ref{inicio} and Proposition \ref{cociente}, so let us go for the general case assuming that we know $K_{n-1}$. Remember that we have the distinguished triangle \ref{trianguloclave}
$$K_n\longrightarrow \pi_{2,+}(\pi_2\phi_n)^+K_{n-1} \longrightarrow M_n\lra.$$
Let us calculate its long exact sequence of cohomology. Using that $(\pi_2\phi)^+$ is an exact functor of $\cD$-modules and Lemma \ref{descdirec}, we get an exact sequence for any $i$ of the form
\begin{equation}\label{descKn}
0\ra\cH^0\pi_{2,+}(\pi_2\phi_n)^+\cH^i K_{n-1} \ra\cH^i\pi_{2,+}(\pi_2\phi_n)^+K_{n-1}\ra \cH^{-1}\pi_{2,+} (\pi_2\phi_n)^+\cH^{i+1}K_{n-1}\ra0.
\end{equation}
If any of the cohomologies of $\cH^iK_n$ is formed, among others, by an extension of $\cO_{\gm}^a$ by $\cO_{\gm}^b$ for certain $a$ and $b$, it will necessarily be trivial (although that does not happen in general; consider for instance the extension $0\ra\cO_{\gm}\ra\cD_{\gm}/(D^2)\ra\cO_{\gm}\ra0$). Indeed, by Theorem \ref{teorema1}, the cohomologies $\cH^ij^+\bar{K}_n$ do not have any singularity at the origin. Now take into account that those cohomologies are nothing but the image by $\iota_n^+$ of those of $K_n$, so any of the extensions of $\cO_{\gm}^a$ that we could have may be extended to the analogous over an affine line and then must be trivial, for $\Ext_{\cD_{\AA^1}}^1(\cO_{\AA^1},\cO_{\AA^1})=0$.

Take $i\in\{-(n-2),\ldots,-1\}$. We know that $\cH^iK_{n-1}\cong\bigoplus_{a=1}^{e_{n-1}}\cK_{a/e_{n-1}}^{\binom{n-1}{i+n-2}}$. By Lemma \ref{Kummerfallo} above,
\begin{equation}\label{kummerKn}
\pi_{2,+}(\pi_2\phi_n)^+\bigoplus_{a=1}^{e_{n-1}}\cK_{a/e_{n-1}}\cong\cO_{\gm}[1]\oplus\cO_{\gm} \oplus\bigoplus_{a=1}^{e_{n-1}}\cK_{a/e_{n-1}}.
\end{equation}

Therefore, we can claim that the exact sequences of the form \ref{descKn} will split for any $i\neq0,-1$. Using that splitting and formulas \ref{kummerKn}, we can be more concrete when writing the long exact sequence of triangle \ref{trianguloclave}, namely
\begin{equation}\label{suclarga}
\begin{array}{c}
\displaystyle 0\ra\cH^{-n}K_n\ra0\ra0\ra\cH^{-(n-1)}K_n\ra\cO_{\gm}\ra0\ra\\
\displaystyle \vdots\\
\displaystyle \ra\cH^iK_n\ra\cO_{\gm}^{\binom{n}{i+n-1}}\oplus \bigoplus_{a=1}^{e_{n-1}}\cK_{a/e_{n-1}}^{\binom{n-1}{i+n-2}}\ra \bigoplus_{a=1}^{e_{n-1}}\cK_{a/e_{n-1}}^{\binom{n-1}{i+n-2}}\ra\\
\displaystyle \vdots\\
\displaystyle \ra\cH^{-1}K_n\ra\cH^{-1}\pi_{2,+}(\pi_2\phi_n)^+K_{n-1}\ra \bigoplus_{a=1}^{e_{n-1}}\cK_{a/e_{n-1}}^{\binom{n-1}{n-3}}\ra\\
\displaystyle\ra\cH^0K_n\ra\cH^0\pi_{2,+}(\pi_2\phi_n)^+K_{n-1} \ra\bigoplus_{a=1}^{e_{n-1}}\cK_{a/e_{n-1}}^{n-2} \oplus\bigoplus_{a=1}^{d_{n-1}}\cK_{a/d_{n-1}}\ra0.
\end{array}
\end{equation}

Then $\cH^{-n}K_n=0$, $\cH^{-(n-1)}K_n\cong\cO_{\gm}$ and the $G_i$ will be a sum of Kummer $\cD$-modules for $i\leq-2$. Denote by $h_i^j$ the dimension of the $j$-th global de Rham cohomology of $G_i$; since
$$h_i^{-1}=\dim\Gamma\left(\gm,\bR\cH om_{\cD_{\gm}}\left(\cO_{\gm},G_i\right)\right),$$
all of the $h_i^{-1}$ must vanish by definition of the $G_i$, and the $h_i^0$ will vanish too for $i\leq-2$ because those $G_i$ are sums of Kummer $\cD$-modules. Notice now that the global de Rham cohomologies of $K_n$ and $\cO_{Z_n}$ are the same. The latter is already known thanks to \cite[Prop. 5.2]{CaMV} because $Z_n$ is the complement of an arrangement of hyperplanes in general position, so we will have the following system of equations:
$$\left\{\begin{array}{l}
c_i+c_{i+1}=\binom{n+1}{i+n}, \,i=-n,\ldots,-2\\
h_{-1}^0+c_{-1}+c_0=\binom{n+1}{n-1}\\
h_0^0+c_0=n+1\end{array},\right.$$
From these equations we get that for $i=-(n-2),\ldots,-1$, $c_{-i}=\binom{n}{i+n-1}$, so that our system can be reduced to:
$$\left\{\begin{array}{l}
h_{-1}^0+c_0=n\\
h_0^0+c_0=n+1\end{array}.\right.$$
Nevertheless, the first point of Proposition \ref{cociente} tells us that $c_0\geq n$, so the equality holds indeed and then $h_{-1}^0=0$ and $h_0^0=1$. Following an analogous argument, from the calculation of the $c_i$ we deduce that $\operatorname{rk}\cH^iK_n\geq\binom{n}{i+n-1}$ at every degree; applying now the second point of Proposition \ref{cociente} the equality holds again for $i\leq-2$ and so $G_i=0$ for those values of $i$. Then we can also claim that every row of the long exact sequence \ref{suclarga} but the last two is a single short exact sequence in itself, all of them with the zero module at the beginning and the end.

This proves everything (so that $\cG_n^{\text{ss}}=G_0$) except the vanishing of $G_{-1}$. We know that its rank is one at most, thanks to point ii of Proposition \ref{cociente}. Then it will be a Kummer $\cD$-module, for $h_{-1}^{-1}=h_{-1}^0=0$. Let us focus first in knowing better $\cH^{-1}\pi_{2,+}(\pi_2\phi_n)^+\cG_{n-1}$.

From the exact sequence
$$0\lra\cG_{n-1}\lra \cH^{0}K_{n-1}\lra\bigoplus_{a=1}^{e_{n-1}}\cK_{a/e_{n-1}}^{n-1}\lra0$$
and diagram \ref{descKn} we can deduce that its nonconstant composition factors, together with some Kummer $\cD$-modules, are composition factors of $\cH^{-1}\pi_{2,+}(\pi_2\phi_n)^+K_{n-1}$, which finds itself in an exact sequence between $\cH^{-1}K_n$ and a sum of Kummer modules. Therefore, $\cH^{-1}\pi_{2,+}(\pi_2\phi_n)^+\cG_{n-1}$ is also an extension of some Kummer $\cD$-modules. By applying the functor $\pi_{2,+}(\pi_2\phi_n)^+$ to any of the exact sequences from \ref{sucesionescoro}, we can claim that it seats in an exact sequence between $\cH^{-1}\pi_{2,+}(\pi_2\phi_n)^+\cH_{n-1}\oplus\bigoplus_{\alpha\in C'_{n-1}}\cH^{-1}\pi_{2,+}(\pi_2\phi_n)^+\cK_\alpha$ and $\bigoplus_{\alpha\in C''_{n-1}}\cH^{-1}\pi_{2,+}(\pi_2\phi_n)^+\cK_\alpha$, exchanging $C'_{n-1}$ and $C''_{n-1}$ if necessary. Let us study then those other objects, and start with the hypergeometric $\cD$-module $\cH_{n-1}$.

Write $\cH_{n-1}=\cH_\gamma(\alpha_i;\beta_j)$ for certain $\gamma,\alpha_i,\beta_j$ and $i,j=1,\ldots,r$ (we do not need their concrete expression now) and call $g:=z(1-z)$. Then, we want to know the kernel of $\dd_z$ (acting by left multiplication) at $\cD_{\gm}\left[z,g^{-1}\right]\langle\dd_z\rangle/(P_0,P_1)$, where $$P_0=\dd_z-\frac{d_nz-w_n}{z(1-z)}D_{\lambda} \text{ and } P_1=\gamma z^{w_n}(1-z)^{d_{n-1}} \prod_{i=1}^r(D_{\lambda}-\alpha_i) -\lambda\prod_{j=1}^r(D_{\lambda}-\beta_j).$$

Dividing by $P_0$, every element of $\cD_{\gm}\left[z,g^{-1}\right]\langle\dd_z\rangle$ can be written as $a=\sum_{i=0}^Na_iD_{\lambda}^i$, with every coefficient $a_i\in\cO_{\gm}\left[z,g^{-1}\right]$. Moreover, we can take $N=r-1$. If $N\geq r$, we have that
$$\dd_z\cdot a=\sum_{i=0}^N\dd_z(a_i)D_{\lambda}^i+\sum_{i=0}^N(d_nz-w_n) g^{-1}a_iD_{\lambda}^{i+1},$$
so taking the symbols with respect to $D_{\lambda}$,
$$(d_nz-w_n)g^{-1}a_N=x\left(\gamma z^{w_n}(1-z)^{d_{n-1}}-\lambda\right)$$
for some $x$, and so,
$$a_N=y\left(\gamma z^{w_n}(1-z)^{d_{n-1}}-\lambda\right)$$
for some $y$. Therefore, $a=a'+yD_{\lambda}^{N-r}P_1$, where $\deg_{D_{\lambda}}a'<N$ and such that $\dd_za$ belongs to the ideal $(P_0,P_1)$ if and only if $\dd_za'$ does.

Take then $a=\sum_{i=0}^{r-1}a_iD_{\lambda}^i$, and let us assume that $\dd_za\in(P_0,P_1)$. As above, there will exist some $x$ and $y$ belonging to $\cO_{\gm}\left[z,g^{-1}\right]$ such that
$$(d_nz-w_n)g^{-1}a_{r-1}=x\left(\gamma z^{w_n}(1-z)^{d_{n-1}}- \lambda\right)\,\text{ and } a_{r-1}=y\left(\gamma z^{w_n}(1-z)^{d_{n-1}}- \lambda\right).$$
Since $\deg_{D_{\lambda}}\dd_za=r$, we necessarily have that $\dd_za=xP_1$. But $\cH^{-1}\pi_{2,+}(\pi_2\phi_n)^+\cH_{n-1}$ is an extension of Kummer $\cD$-modules: either it is a subobject of $\cH^{-1}\pi_{2,+}(\pi_2\phi_n)^+\cG_{n-1}$ or it lies in between it and some Kummer $\cD$-modules according to the two possibilities in \ref{sucesionescoro}, so there must exist some $\alpha$ such that $(D_{\lambda}-\alpha)a=y'P_1$.

Now $\dd_za$ and $(D_{\lambda}-\alpha)(d_nz-w_n)g^{-1}a$ share their leading term in $D_{\lambda}$, so
$$\left(\dd_z-(D_{\lambda}-\alpha)(d_nz-w_n)g^{-1}\right)a$$
will vanish. Therefore, for every $i=0,\ldots,r-1$ we will have that
$$\dd_z(a_i)-(d_nz-w_n)g^{-1}(D_{\lambda}-\alpha)(a_i)=0.$$
Writing locally each $a_i$ as $\sum_{j\in\ZZ}a_{ij}\lambda^j$, with each $a_{ij}$ belonging to $\kk\left[z,g^{-1}\right]$, we must have that, for any $i$ and $j$,
$$\dd_z(a_{ij})=(d_nz-w_n)g^{-1}(j-\alpha)a_{ij}.$$
This differential equation has as formal solution space the $\kk$-span of $\left(z^{w_n}(1-z)^{d_{n-1}}\right)^{\alpha-j}$, which is algebraic only if $w_n\alpha,d_{n-1}\alpha\in\ZZ$. In conclusion, $\cH^{-1}\pi_{2,+}(\pi_2\phi_n)^+\cH_{n-1}$ is an extension of certain $\cK_\alpha$ for $w_n\alpha$ and $d_{n-1}\alpha$ being an integer.

Let us turn our attention now to the Kummer side of $\cG_{n-1}$. Thanks to Lemma \ref{Kummerfallo}, the $\cD$-module $\cH^{-1}\pi_{2,+}(\pi_2\phi_n)^+\cK_\alpha$ does not vanish if and only if $w_n\alpha,d_{n-1}\alpha\in\ZZ$. In conclusion, gathering the information of the last paragraphs, we know that $\cH^{-1}\pi_{2,+}(\pi_2\phi_n)^+\cG_{n-1}$ must be an extension of Kummer $\cD$-modules $\cK_\alpha$ for some $\alpha$ such that $w_n\alpha$ and $d_{n-1}\alpha$ are always integer numbers.

Point iii of Proposition \ref{cociente} tells us that $G_{-1}$, being a Kummer $\cD$-module $\cK_\alpha$, must verify that $d_n\alpha$ is an integer. From the second row from the last of the long exact sequence \ref{suclarga} we can affirm that it is also a subobject of either $\cH^{-1}\pi_{2,+}(\pi_2\phi_n)^+\cG_{n-1}$ or $\bigoplus_{a=1}^{e_{n-1}}\cK_{a/e_{n-1}}^{\binom{n-1}{n-3}}$. In any case we would have that $d_{n-1}\alpha\in\ZZ$, so $w_n\alpha$ should be an integer as well.

However, $G_{-1}$ is the nonconstant composition factor of $\cH^{-1}K_n$ (if there is any), and we are imposing a condition on $\alpha$ that depends heavily on the order we have followed to perform the induction, whereas $K_n$ does not change when changing the variables and thus the inductive steps. Renaming the variables so that we exchange $w_n$ and any other $w_i$ and following the same inductive procedure, we obtain that $G_{-1}$ must actually be a nontrivial Kummer $\cD$-module $\cK_\alpha$ such that $w_i\alpha\in\ZZ$ for every $i=0,\ldots,n$. But if such a thing happens, then $e_n\alpha$ should also be an integer, and by assumption $e_n=1$, so in the end $G_{-1}=0$ and we have finished the proof.
\end{proof}

\begin{nota}\label{sucesionH0}
By applying formula \ref{descKn} and looking at the second row from the last of the long exact sequence \ref{suclarga} we can also deduce from the proof of the proposition that $\cH^{-1}\pi_{2,+}(\pi_2\phi_n)^+\cG_{n-1}=0$ and $\cH^{-1}\pi_{2,+}(\pi_2\phi_n)^+\cH^0K_{n-1}\cong\cO_{\gm}^{n-1}$. Consequently, we have as well an exact sequence which will be very useful in the following result, formed by the last row of the long exact sequence \ref{suclarga}:
$$0\lra\cH^0K_n\lra\cH^0\pi_{2,+}(\pi_2\phi_n)^+K_{n-1}\lra\bigoplus_{a=1}^{e_{n-1}}\cK_{a/e_{n-1}}^{n-2} \oplus\bigoplus_{a=1}^{d_{n-1}}\cK_{a/d_{n-1}}\lra0.$$
\end{nota}

What we still have to prove is the rest of the properties of $\cG_n$ that are of interest to us: its generic rank, its singular points and its exponents at the origin and infinity. The first two can still be obtained within the inductive context.

\begin{prop}\label{rksing}
For every $n\geq1$, the generic rank of $\cG_n$ is $d_n-e_n$, and it has a unique singularity at $\gm$, namely at $\gamma_n$.
\end{prop}
\begin{proof}
We already know that $\cG_n$ is a regular holonomic $\cD_{\gm}$-module, and its Euler-Poincaré characteristic is $-1$, so by Proposition \ref{hyperLS} it will have a singularity at some point $\lambda_0$. Its restriction to the rest of $\gm$ will then be a module with integrable connection of some rank to be determined.

Since we know by Lemma \ref{inicio} that the statement of the proposition is true for $n=1$, let us prove it for a general $n$ by induction, and so let us assume its veracity for lower values of the index $n$. Assume as well that $e_n=1$; as before the general case can be easily deduced from this one.

As we discussed in Remark \ref{sucesionH0}, we had an exact sequence of the form
$$0\lra\cH^0K_n\lra\cH^0\pi_{2,+}(\pi_2\phi_n)^+K_{n-1}\lra\bigoplus_{a=1}^{e_{n-1}}\cK_{a/e_{n-1}}^{n-2} \oplus\bigoplus_{a=1}^{d_{n-1}}\cK_{a/d_{n-1}}\lra0.$$
Consider the long exact sequence of cohomology of the triangle
$$\pi_{2,+}(\pi_2\phi_n)^+\cG_{n-1}\lra\pi_{2,+}(\pi_2\phi_n)^+\cH^0K_{n-1}\lra \pi_{2,+}(\pi_2\phi_n)^+\bigoplus_{a=1}^{e_{n-1}}\cK_{a/e_{n-1}}^{n-1},$$
and focus in degree $-1$. We know that $\cH^{-1}\pi_{2,+}(\pi_2\phi_n)^+\cG_{n-1}$ vanishes and $\cH^{-1}\pi_{2,+}(\pi_2\phi_n)^+\cH^0K_{n-1}\cong\cO_{\gm}^{n-1}$ and, by formula \ref{kummerKn}, that
$$\cH^0\pi_{2,+}(\pi_2\phi_n)^+\bigoplus_{a=1}^{e_{n-1}}\cK_{a/e_{n-1}}^{n-1} \cong\cO_{\gm}^{n-1}[1]\oplus\cO_{\gm}^{n-1}\oplus\bigoplus_{a=1}^{e_{n-1}}\cK_{a/e_{n-1}}^{n-1}.$$
Therefore, we have in fact the exact sequence
$$0\lra\pi_{2,+}(\pi_2\phi_n)^+\cG_{n-1}\lra\cH^0\pi_{2,+}(\pi_2\phi_n)^+\cH^0K_{n-1}\lra \cO_{\gm}^{n-1}\oplus\bigoplus_{a=1}^{e_{n-1}}\cK_{a/e_{n-1}}^{n-1}\lra0.$$
Summing up, the generic rank of $\cG_n$ is equal to $-d_{n-1}+e_{n-1}-1$ plus the generic rank of $\pi_{2,+}(\pi_2\phi_n)^+\cG_{n-1}$; let us find that one.

Consider $(\pi_2\phi_n)^+\cG_{n-1}$. It is a regular holonomic $\cD$-module over $(\gm-\{1\})\times\gm$, having singularities along the curve $C_{\lambda}:=(\pi_2\phi_n)^{-1}(\gamma_{n-1})$ with equation $\lambda=\gamma_{n-1}z^{w_n}(1-z)^{d_{n-1}}$. Fixing $\mu$ in $\gm$, the intersection of $C_\lambda$ and the line $\lambda=\mu$ is a set $C_{\mu}$ of $d_n$ points whenever $\mu\neq\gamma_n$; it is formed by $d_n-1$ points otherwise.

Now, for any value $\mu$ of $\lambda$, consider the cartesian diagram
$$\xymatrixcolsep{4pc}\xymatrix{\ar @{} [dr] |{\Box} (\gm-\{1\})\times\{\mu\} \ar[d]_{\pi_2} \ar[r]^{\id\times i_{\mu}} & (\gm-\{1\})\times\gm \ar[d]^{\pi_2}\\ \{\mu\} \ar[r]^{i_{\mu}} & \gm}.$$
By applying base change,
$$i_{\mu}^+\pi_{2,+}(\pi_2\phi_n)^+\cG_{n-1}\cong \pi_{2,+}(\id\times i_{\mu})^+(\pi_2\phi_n)^+\cG_{n-1}.$$

Assume that $\mu\neq\lambda_0$. Since $\pi_{2,+}(\pi_2\phi_n)^+\cG_{n-1}$ has no singularities at $\mu$, $i_{\mu}^+\pi_{2,+}(\pi_2\phi_n)^+\cG_{n-1}$ is a single $\kk$-vector space, and its dimension equals the generic rank of $\pi_{2,+}(\pi_2\phi_n)^+\cG_{n-1}$.

On the other hand, the morphism $\pi_2\phi_n(\id\times i_{\mu})$ is étale, so the complex $(\id\times i_{\mu})^+(\pi_2\phi_n)^+\cG_{n-1}$ is actually a single $\cD_{\gm-\{1\}}$-module of rank $d_{n-1}-e_{n-1}$. Its Euler-Poincaré characteristic will be
$$-\dim\cH^0\pi_{2,+}(\id\times i_{\mu})^+(\pi_2\phi_n)^+\cG_{n-1}=-\dim i_{\mu}^+\pi_{2,+}(\pi_2\phi_n)^+\cG_{n-1}= -\rk\pi_{2,+}(\pi_2\phi_n)^+\cG_{n-1}.$$

Recall that $\cG_{n-1}$ has no punctual part, so neither $(\id\times i_{\mu})^+(\pi_2\phi_n)^+\cG_{n-1}$ has. Then it will be the intermediate extension of its restriction to $(\gm-\{1\})-C_{\mu}$ and applying \cite[Thm. 2.9.9]{KaESDE},
$$\chi\left((\id\times i_{\mu})^+(\pi_2\phi_n)^+\cG_{n-1}\right)=-(d_{n-1}-e_{n-1})-|C_{\mu}|=-(d_n+d_{n-1}-e_{n-1}),$$
whenever $\mu\neq\gamma_n$, where the last summand is a consequence of [\textit{op. cit.}, Lem. 2.9.7] applied to $\cH_{n-1}$. In conclusion, the generic rank of $\cG_n$ must be $d_n+d_{n-1}-e_{n-1}-d_{n-1}+e_{n-1}-1=d_n-1$. Now, if the point $\lambda_0$ where $\cG_n$ has a singularity were different from $\gamma_n$, then we could do the same process above and see that
$$\chi\left(i_{\mu}^+(\pi_2\phi_n)^+\cG_{n-1}\right)= -(d_n+d_{n-1}-e_{n-1}-1),$$
which cannot be possible. Therefore, $\lambda_0=\gamma_n$.
\end{proof}

\section{Exponents of $\cG_n$ and proof of the main result}\label{fourier}

In the previous section we proved every part of the main theorem that could be accomplished by the inductive process; we will provide in this one the proof of Theorem \ref{teorema2} by using all the information given by the results of section \ref{induccion} and finding the values for the exponents of $\cG_n$ at both the origin and infinity.

\vspace{.3cm}

\noindent\textit{Proof of Theorem \ref{teorema2}.} Assume again that $\gcd(w_0,\ldots,w_n)=1$. The statement about the constant part of $K_n$ and the existence of $\cG_n$ and the exact sequence
$$0\lra\cG_n\lra\cH^0K_n\lra\cO_{\gm}^n\lra0$$
is Proposition \ref{constkn}.

Now by Proposition \ref{rksing}, $\cG_n$ is a regular $\cD_{\gm}$-module of Euler-Poincaré characteristic $-1$, of rank $d_n-1$ and singularities at the origin, $\gamma_n$ and infinity, so by Propositions \ref{charcero}, \ref{isohyp} and \ref{hyperLS} its semisimplification will consist of $k$ Kummer $\cD$-modules and an irreducible hypergeometric $\cD$-module of type $(d_n-1-k,d_n-1-k)$, with a singularity at $\gamma_n$, that is to say,
$$\cG_n^{\text{ss}}=\bigoplus_{\alpha\in A}\cK_{\alpha}\oplus\cH_n,$$
where $|A|=k$ and $\cH_n$ is an irreducible hypergeometric $\cD$-module of type $(d_n-1-k,d_n-1-k)$ of the form $\cH_{\gamma_n}(\alpha_i;\beta_j)$.

Since we want to characterize the Kummer $\cD$-modules and $\cH_n$, we only need, by virtue of Propositions \ref{exponentes} and \ref{isohyp}, to find the exponents of $\cG_n$ at both zero and infinity. Thus those occurring at both points will determine the Kummer summands, and the rest will be the parameters of $\cH_n$ (cf. \cite[Cor. 3.7.5.2]{KaESDE}).

By Remark \ref{semisimple}, $\bar{K}_n$ cannot have a singularity at the origin. Since $j^+\bar{K}_n=\iota_n^+K_n$ by construction, the $d_n-1$ exponents at the origin of $\cG_n$ can only be fractions of denominator equal to $d_n$, say with numerators $k_1,\ldots,k_{d_n-1}\in\{1,\ldots,d_n\}$.

In fact, we can deduce from Theorem \ref{interesante} that the exponents both at the origin and infinity should be, up to a parameter $b$ coming from tensoring with a Kummer $\cD$-module $\cK_{b/d_n}$, all of the exponents but one of the hypergeometric $\cD_{\gm}$-module
$$\cH_{\gamma_n} \left(\cancel\left(\frac{1}{w_0},\ldots,\frac{w_0}{w_0},\ldots, \frac{1}{w_n},\ldots,\frac{w_n}{w_n}; \frac{1}{d_n},\ldots,\frac{d_n}{d_n}\right)\right).$$
The one missing will be of the form $a/d_n$, because we know that the rank of $\cG_n$ is $d_n-1$ thanks to Proposition \ref{rksing}. Summing up, a list of the exponents (of some representative of them in $\kk$ in fact) is, respectively at the origin and infinity,
$$W_n:=\left(\frac{j}{w_i}+\frac{b}{d_n}\,\text{ : }\, i=0,\ldots,n,j=1,\ldots,w_i\right) -\left(\frac{a+b}{d_n}\right),$$
$$\left(\frac{k}{d_n}\,\text{ : }\, k=1,\ldots,d_n\right) -\left(\frac{a+b}{d_n}\right).$$
Now we ought to see which possible values of $a$ and $b$ can really occur. We will see that we can take $a$ and $b$ to be both of them $d_n$ by using the work done at \cite[\S4]{CaKoszul}.

From [\textit{op. cit.}, Thm. 1.2] we know that the exponents at the origin of $\cH^0\lambda_{n,+}\cO_{Z_n}$ can only be those of the form $j/w_i\!\mod\ZZ$, and each of them which is not integer occurs with the same multiplicity (without counting coincidences among some $j/w_i$ for different values of $j$ and $i$).

Let us show that the multiplicity at the origin of one is $n$ at least. Indeed, assume that some $w_i$ is strictly bigger than 1 and that the fixed multiplicity for each noninteger expression $j/w_i$ is some $m>1$. Then there will exist a $\eta_0=j_0/w_{i_0}$ for some admissible choice of $i_0\in\{0,\ldots,n\}$ and $j_0\in\{1,\ldots,w_{i_0}\}$ such that there will be $m$ different admissible pairs $(i_1,j_1),\ldots,(i_m,j_m)$ with $\eta_0=j_k/w_{i_k}+b/d_n$ for each $k=1,\ldots,m$. But then this implies that all those $j_k/w_{i_k}$ are the same number, so we must be able to find an exponent $\eta_1=j_1/w_{i_1}$ with multiplicity $m^2$. Therefore, we can find $m^2$ more admissible pairs of the form $(i'_1,j'_1),\ldots,(i'_{m^2},j'_{m^2})$ such that $\eta_1=j'_k/w_{i'_k}+b/d_n$ for each $k=1,\ldots,m^2$. As a consequence, all of the $j'_k/w_{i'_k}$ coincide and so there must be an exponent $\eta_2=j'_1/w_{i'_1}$ of multiplicity $m^3$.

If we continue this argument, we must reach a point in which there are no more new admissible pairs of the form $(i,j)$, so since the exponents are not all of them equal (because $w_i>1$ for certain $i=0,\ldots,n$), necessarily $\eta_0=\eta_l$ for some $l>0$, but this is a contradiction since they had different multiplicities. In conclusion, either the multiplicity of each noninteger exponent is one, so that the multiplicity of one is $n$, or the only exponent is one and so its multiplicity is $d_n-1$.

That implies that there is no possible choice of $a$ and $b$ other than $a=b=d_n$. Indeed, if $a\neq d_n$, then the value $b/d_n$ appears at $W_n$ with multiplicity $n+1$, and if it is not $1$, then there must exist $n+1$ different $j_k<w_k$ such that $j_k/w_k=j_l/w_l=b/d_n$, but this is impossible because $\gcd(w_i)=1$ (cf. \cite[Lem. 4.2]{CaKoszul}). Therefore, $b=d_n$, but in that case, since $a\neq d_n$, there would exist an exponent at both the origin and infinity equal to $1$, and then $\cG_n$ would have a composition factor equal to $\cO_{\gm}$, which cannot happen by Proposition \ref{constkn}.

Thus $a=d_n$. Then if $b\neq d_n$ there must exist again an exponent at both the origin and infinity equal to $1$, because its multiplicity at both points is positive, but as before this is a contradiction, so in conclusion $a=b=d_n$.

Once we have proved everything above for $e_n=1$, the statement for a bigger $e_n$ follows easily taking into account that the direct image $[e_n]_+$ of a irreducible hypergeometric $\cD_{\gm}$-module remains both irreducible and hypergeometric and the exponents of the direct image are all the classes of $\kk/\ZZ$ such that multiplying by $e_n$ we recover the original ones. With this last step we can finish the inductive step and use it in the results of section \ref{induccion}.

Last, the exact sequence of $\cG_n$ is corollary \ref{corosucGn}, once we have proved everything else in Theorem \ref{teorema2} about $K_r$ for $r=n,n+1$.\hfill$\Box$

\begin{nota}
One could wonder when $K_n$ or $\bar{K}_n$ have integer exponents (or, equivalently in the complex analytical setting, unipotent formal local monodromy) at the origin and infinity, respectively. This is a remarkable property (and even rare, in the case of $K_n$), and can be characterized in terms of the $w_i$. The proofs can be found, respectively, in \cite[Lem. 8.8]{KaAnother} and \cite[Rem. 2.2]{DouSab}.
\end{nota}

\bibliographystyle{amsalpha}
\bibliography{DworkTesis}

\end{document}